\documentclass{article}
\usepackage[utf8]{inputenc}
\usepackage{authblk}
\usepackage{setspace}
\usepackage[margin=1.25in]{geometry}
\usepackage{graphicx}
\graphicspath{ {./figures/} }
\usepackage{subcaption}
\usepackage{amsmath}
\usepackage{amssymb}
\usepackage{lineno}
\usepackage{comment}
\usepackage{amsthm}
\newtheorem{theorem}{Theorem}
\usepackage{amsfonts} 

\usepackage{url}
\usepackage{hyperref}
\usepackage{xcolor}
\hypersetup{
    colorlinks,
    linkcolor={red!50!black},
    citecolor={blue!50!black},
    urlcolor={blue!80!black}
}
\usepackage[natbib=true,backend=biber,sorting=none,style=ieee, citestyle=numeric-comp]{biblatex}

\usepackage{pdfpages}

{\title{\begin{center} A Comprehensive and Detailed Within-Host Modeling Study involving crucial biomarkers  and Optimal Drug regimen for  Type - I Lepra Reaction : A Deterministic Approach  \end{center}}}

\author[1]{Dinesh Nayak}
\author[2]{Bishal Chhetri}
\author[3]{D. K. K. Vamsi}
\author[4]{Swapna Muthusamy}
\author[5]{Vijay M. Bhagat}

\affil[1, 2, 3]{ \ Department of Mathematics and Computer Science, Sri Sathya Sai Institute of Higher Learning, India.}
\affil[4,5]{Central Leprosy Teaching and Research Institute - CLTRI, Chennai, India.}
\affil[1]{First Author. Email: dineshnayak@sssihl.edu.in}
\affil[3]{Corresponding author. Email: dkkvamsi@sssihl.edu.in}

\date{August 2021}

\begin{document}

\maketitle

\begin{abstract}
  Leprosy (Hansen's disease) is an infectious, neglected tropical disease caused by the {\textit{Mycobacterium Leprae (M. Leprae).}} Each year there are approximately $ 2, 02,189 $ new cases are detected globally. In the year 2017 more than half million people were disabled due to leprosy and almost 50000 new cases are added every year world wide. In leprosy, lepra reactions are the major cause for nerve damage leading to disability. Early detection of lepra reactions through study of biomarkers have important role in prevention of subsequent disabilities.  To our knowledge there seems to be very limited literature available on within-host modeling at cellular level involving   the crucial biomarkers  and the possible optimal drug regimen for leprosy disease and lepra reactions. Motivated by these observations, in this study, we have proposed and analyzed a three dimensional mathematical model to capture the dynamics of susceptible schwann cells, infected schwann cells and the bacterial load based on the pathogenesis of leprosy. We initially have  established the existence of solution and later validated the model through the disease characteristics of leprosy.  Further we dealt with the  local and global stability of different equilibria   about the reproduction number value $\mathcal{R}_{0} = 1$. Later for numerical studies  we estimated the parameters from various clinical papers to make the model more practical. The sensitivity of couple of parameters was evaluated through Partial Rank Correlation Coefficient  (PRCC) method to find out the single most influential parameter and also combination of two most influential parameters was studied using Spearman's Rank Correlation Coefficient  (SRCC) method.  The sensitivity of other remaining parameters was evaluated using Sobol's index. We then have framed and studied an optimal control problem considering  the different medication involved in the Multi Drug Therapy (MDT) as control variables. We further studied  this optimal control problem along with both MDT and  steroid interventions. Finally we did the comparative and effectiveness study of these different control interventions. The  finding from this novel and comprehensive study will help the clinicians and public health researchers involved in the process of elimination and  eradication of  leprosy. 
 \end{abstract}

\section*{Keywords} \vspace{.25cm}

Hansen's disease; type - I  lepra reaction ; PRCC method; SRCC method; Sobol's Index; MDT;  Comparative and effectiveness study

\section{Introduction}

\vspace{.2cm}

\hspace{0.3 in}Leprosy  is an infection caused by
slow-growing bacteria called {\textit{ Mycobacterium leprae.}} Leprosy
is also known as  {Hansen disease} and it is considered
to be the oldest disease known to humans. Primarily
the bacteria  affects the skin and peripheral nerves of the
host body. In some of the cases it affects the the mucosa
of the upper respiratory tract and the eyes. According to the  WHO report \cite{world2020global}, global annual number of new cases detected in 2019 was about 2, 02,189.  In the year 2017 more than half million people were disabled due to leprosy and almost 50000 are added
every year world wide. In leprosy, lepra reactions are the major cause for nerve damage leading to disability. Early detection of lepra reactions through study of biomarkers have important role in prevention of subsequent disabilities.  \\

\hspace{0.3 in} During the course of the leprosy disease there can be  sudden changes in immune-mediated response to Mycobacterium leprae antigen which are referred to as leprosy (lepra) reactions. The reactions manifest as acute inflammatory episodes rather than chronic infectious course. There are mainly two types of leprosy reactions. Type 1 reaction is associated with cellular immunity and particularly with the reaction of T helper 1 (Th1) cells to mycobacterial antigens. This reaction involves exacerbation of old lesions leading to the erythematous appearance. Type 2 reaction or erythema nodosum leprosum (ENL) is associated with humoral immunity. It is characterized by systemic symptoms along with new erythematous subcutaneous nodules.  \\

\hspace{0.3 in} Several clinical and experimental studies has been done on Leprosy. Some
works deal about the growth of the \textit{M. Leprae} \cite{ojo2022mycobacterium}, some on
pathogenesis \cite{massone2022pathogenesis}. Now in the context of the
mathematical modeling of the disease, there are some contributions that explore the dynamics of
transmission  of leprosy at population level \cite{blok2015mathematical}. In \cite{giraldo2018multibacillary} the 
transmission dynamics of the multibacillary leprosy (MB) and paucibacillary leprosy (PB)
including a delay is dealt with. Some works dealing with the  cellular level  dynamics is explored in   \cite{ghosh2021mathematical}. To our knowledge as of date  there is no work done yet to explore the dynamics at the level of  bio-markers and also there seems to be no mathematical literature available dealing with the optimal drug regimen for treating leprosy and lepra reactions. A mathematical modeling study to this extent will help the clinicians to dissemination of the leprosy by targeting the crucial biomarkers  with minimal damage and also helps them for the optimal drug regimen. \\

\hspace{0.3 in} Motivated by the above observations,  in this study we have proposed and
analyzed an within-host  three dimensional mathematical model to capture the dynamics of susceptible schwann cells, infected schwann cells and the bacterial load involving the causation biomarkers  for type - I  lepra reaction  based on a detailed flow chart dealing with the pathogenesis of leprosy devloped from the clinical works \cite{ridley2013pathogenesis,sasaki2001mycobacterium, weddell1963pathogenesis}. We initially study the natural history of the disease followed studies on the optimal drug regimen for type - I  lepra reaction. \\  

\hspace{0.3 in} The section wise division of this article is as follows.  In section 2 we formulate the mathematical model dealing with the type - I  lepra reaction  based on the pathogenesis of type - I  lepra reaction . Later in section 3 we establish the existence, positivity and boudedness of the developed model followed by the  local and global stability of different equilibria about the reproduction number value $\mathcal{R}_{0} = 1$ followed by bifurcation analysis. Further in section 4 we numerically depict the theoretical findings of section 3. We validate the proposed model via the leprosy disease characteristics using 2D heat plots in section 5. Further in section 6 we perform the sensitivity analysis of the model  parameters. Later in section 7  we do the  optimal control studies  considering the different
medication involved in the Multi Drug Therapy (MDT) as control variables followed by  optimal
control studies involving both  MDT and steroid interventions. Finally we do the comparative and effectiveness study of these different control interventions in section 8. We do the discussion and conclusion in section 9.    

\includepdf[pages=-,pagecommand={},width=0.9\textwidth]{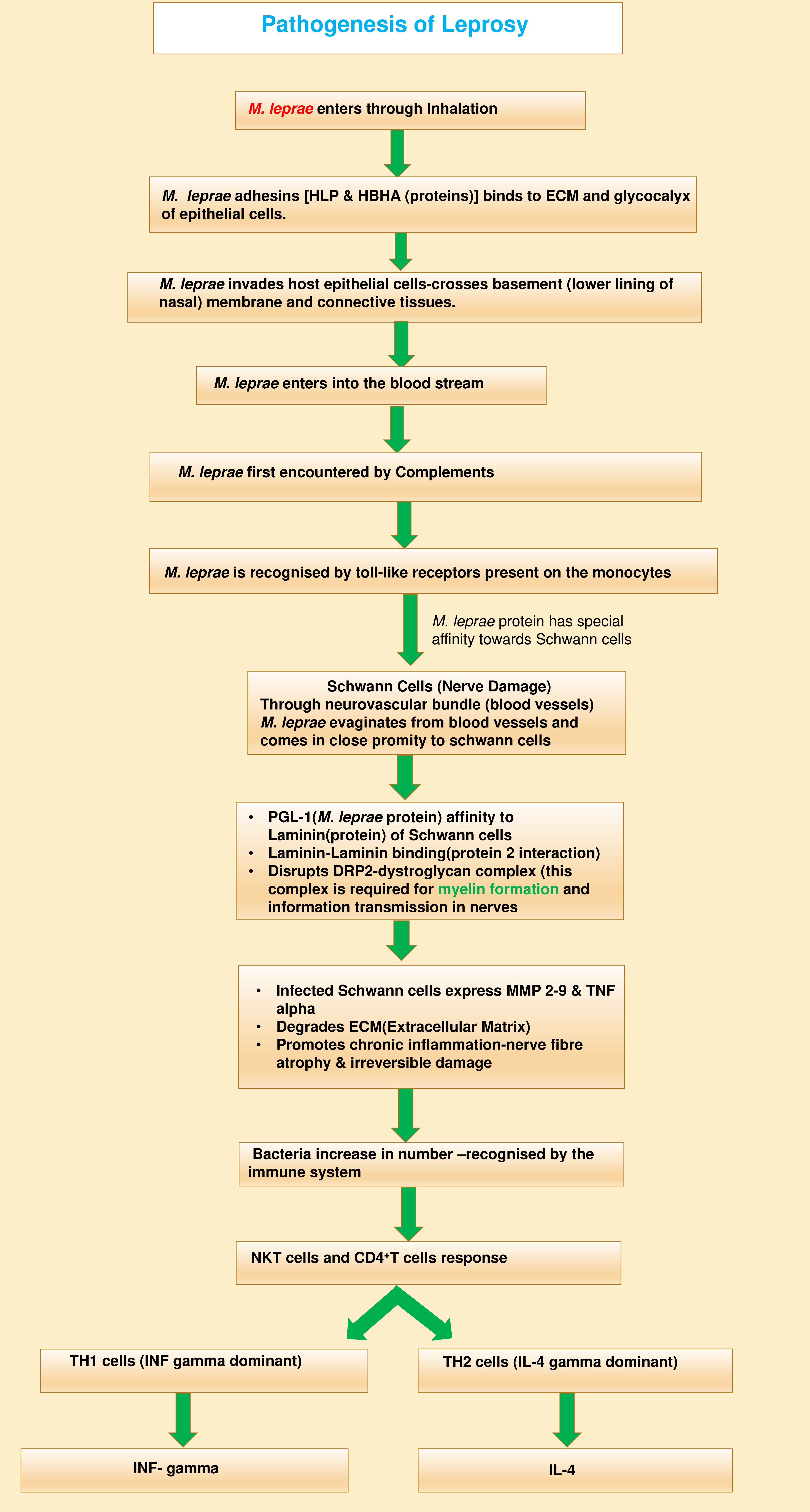}

\section{mathematical model formulation} \vspace{.2cm}

\hspace{0.3 in} Based on the pathogenesis of leprosy dealt in the flow chart earlier  we consider a three compartment model dealing with Susceptible  schwann cells $S(t)$, Infected schwann cells  $I(t)$ and the Bacterial load $B(t)$. We have taken the help of system of ODE's  to interpret the biological dynamics in term of mathematical equations.  \\

\hspace{0.3 in} The dynamics of the susceptible cells i.e. $\frac{dS}{dt}$ will depend on the natural birth rate $\omega$. Also according to the law of mass action the susceptible cell decrease at a rate $\beta$ hence the term $-\beta S B.$   The susceptible cells decrease due the natural death and and the cytokines responses. Next for the dynamics of the infected  cells i.e. $\frac{dI}{dt}$ the infected cells increase by  $\beta S B$ and decrease by  the  natural death and by cytokines responses. The growth of the bacteria depends on the burst rate of the infected cells. Therefore  the compartment $\frac{dB}{dt}$ has $\alpha I$ and the bacterial load decreases due to natural death  of the bacteria and death due to the cytokines. In summary we propose the following within-host model.

\begin{eqnarray}
   	\frac{dS}{dt}& =&  \omega \ - \beta SB  - \gamma S - \mu_{1} S  \label{sec2equ1} \\
   	\frac{dI}{dt} &=& \beta SB \ -\delta I - \mu_{1} I   \label{sec2equ2}\\ 
   	\frac{dB}{dt} &=&  \alpha I  \ - ( d_{11}  + d_{12}  + d_{13}  +  d_{14}+ d_{15} + d_{16} + d_{17} )B    \ -  \mu_{2} B \label{sec2equ3}
\end{eqnarray}

\begin{table}[ht!]   
 \centering 
\begin{tabular}{|l|l|} 
\hline

\textbf{Symbols} &  \textbf{Biological Meaning} \\  
\hline\hline 
$S$ & Susceptible schwann cells  \\
\hline\hline

$I$ & Infected schwann cells  \\
\hline\hline
$B$ & Bacterila load  \\

\hline\hline
$\omega$& Natural birth rate of the  susceptible cells\\

\hline\hline
$\beta$ & Rate at which schwann cells are infected  \\

\hline\hline
$\gamma$& Death rate of the  susceptible cells due to cytokines \\

\hline\hline
$\mu_{1}$ & Natural death rate of  schwann cells and  infected\\
& schwann cells \\

\hline\hline
$\delta$ & Death rate of infected schwann cells  due to cytokines\\

\hline\hline
$\alpha$ & Burst rate of bacterial particles \\

\hline\hline

$d_{11}, \hspace{.25cm} d_{12}, \hspace{.25cm} d_{13}, \hspace{.25cm} d_{14}, \hspace{.25cm} d_{15}, \hspace{.25cm} d_{16},\hspace{.25cm} d_{17}$ & Rates at which {\textit{ M. Leprae }}
is removed\\
&because of the release of cytokines IL-2, IL-7\\
& $TNF-\alpha,  \hspace{.2cm} IFN-\gamma$, IL -12, 
 IL- 15, IL-17  respectively   \\
\hline\hline
$\mu_{2}$ & Natural death rate of {\textit{ M. Leprae }}\\
\hline

\end{tabular}
\end{table} \vspace{.2cm}

\section{Stability Analysis} \vspace{.2cm}

\subsection{{\bf{Positivity and Boundedness}}} 

\begin{theorem}{\bf{Positivity}}:
For the model (\ref{sec2equ1}) - (\ref{sec2equ3}) if initially  $S(0) > 0, I(0) > 0$ and $B(0) > 0$ then for all $t\in [0, t_{0}]$ where $t_{0} > 0$, $S(t),  I(t),  B(t)$ will remain positive in $\mathbb{R}^{3}_{+}$.
\end{theorem}

\begin{proof}

We now aim to show that for all $t\in [0,t_{0}]$, $S(t), I(t)$ and $B(t)$ will be positive in $\mathbb{R}^{3}_{+}$. \\

Consider 
\begin{eqnarray*}
     \frac{dS}{dt}& = &  \omega  - \beta SB  - \gamma S - \mu_{1} S   \\
                 & \geq & - (\beta B  - \gamma  - \mu_{1}) S   \\
\end{eqnarray*}
  On solving the above inequality we get  
\begin{equation*}
     S(t) \geq e^{ - (\gamma t + \mu_{1}t +\int B dt ) }  
                 >  0    \\
\end{equation*}
$$
    \therefore S(t) > 0, \hspace{0.01 in}  \forall t \in [0,t_{0}].
$$
In similar lines we see that
$$\frac{dI}{dt} \geq  -\delta I - \mu_{1} I \implies I(t)   \geq e^{-(\delta + \mu_{1})t}> 0$$
$$\frac{dB}{dt} \geq  - y B - \mu_{2} B \implies B(t)   \geq e^{-(y + \mu_{2})t}> 0$$    
Here $y=( d_{11}  + d_{12}  + d_{13}  +  d_{14}+ d_{15} + d_{16} + d_{17} )$.
Thus for all $ t\in [0,t_{0}]$, $S(t), I(t)$ and $B(t)$ will remain positive i.e. in $\mathbb{R}^{3}_{+}$.
\end{proof}
\begin{theorem}{\bf{Boundedness}}:
There exists an upper bound for each of the variable $S(t), I(t), B(t)$ for all $t \in [0,t_{0}].$

\end{theorem}
\begin{proof}
Let us consider 
\begin{eqnarray*}
\dfrac{dS}{dt}+\dfrac{dI}{dt} &= &\omega - (\gamma+\mu_{1})S -(\delta+\mu_{1})I \\
\implies \dfrac{d(S+I)}{dt} &\leq& \omega -min\{(\gamma+\mu_{1}),(\delta+\mu_{1})\}(S+I)
\end{eqnarray*}
Considering $k=min\{(\gamma+\mu_{1}),(\delta+\mu_{1})\}$ and integrating the above we get, 
\begin{eqnarray*}
(S+I)(t)&\leq & \frac{\omega}{k} + c_{1}e^{-kt}\\
\end{eqnarray*}
Hence $$\limsup\limits_{t\to\infty}{(S+I)}\leq \limsup\limits_{t\to\infty}{\left(\frac{\omega}{k} + c_{1}e^{-kt}\right)} = \frac{\omega}{k}<\infty $$
$\therefore$ $(S+I)(t)$ is bounded thus $S(t), I(t)$ are  bounded. Since 
$$S(t),I(t)\leq (S+I)(t)$$
Now  for $t \in [0,t_{0}]$, there exist $S_{max}$ and $I_{max}$ such that $S(t)\leq S_{max} $, $I(t)\leq I_{max}. $

We now consider \  $ \dfrac{dB}{dt}= \alpha I - (y+\mu_{2})B.$

Solving the above differential equation for $B(t)$, we get,
\begin{eqnarray*}
       B e^{(y+\mu_{2})t} &=& \int \alpha I e^{(y+\mu_{2})t} dt\\
       &\leq& \int \alpha \frac{\omega}{k} e^{(y+\mu_{2})t} dt\\
       &=& \frac{\alpha\omega}{k(y+\mu_{2})}e^{(y+\mu_{2})t}+c_{2}\\
       \implies B(t) &\leq& \frac{\alpha\omega}{k(y+\mu_{2})}+c_{2}e^{-(y+\mu_{2})t}\\
\end{eqnarray*}
$$\therefore  \limsup\limits_{t\to\infty}{B(t)}\leq \limsup\limits_{t\to\infty}\left({\frac{\alpha \omega}{k(y+\mu_{2})}+c_{2}e^{-(y+\mu_{2})t}}\right) = \frac{\alpha \omega}{k(y+\mu_{2})}< \infty $$
Hence there exists an upper bound for  $B(t)$, say $B_{max}$  for  $t\in [0,t_{0}].$

Hence $S(t),I(t)$ and $B(t)$ all are bounded for  $t\in [0,t_{0}].$
\end{proof} 
 
\subsection{{\bf{Existence of the solution}}}
\begin{theorem}
Let $t_{0}>0$.  If the model (\ref{sec2equ1}) - (\ref{sec2equ3})  initially satisfies  $S(0) > 0, I(0) > 0$ and $B(0) > 0$  then $\forall t > 0$ there exists a unique solution for the system  in $\mathbb{R}^{3}_{+}$. 
\end{theorem}
\begin{proof}
 The system  (\ref{sec2equ1}) - (\ref{sec2equ3}) in the vectorial form is given by  
$$\dfrac{dX}{dt}=f(X)$$
where
$$ X=\begin{bmatrix}
S(t)\\
I(t)\\
B(t)
\end{bmatrix} and,
f(X)=\begin{bmatrix}
\omega \ - \beta SB  - \gamma S - \mu_{1} S\\
\beta SB \ -\delta I - \mu_{1} I   \\
\alpha I  \ - y B    \ -  \mu_{2} B
\end{bmatrix}  $$
Now we can see  that $f(X):\mathbb{R}^3 \to \mathbb{R}^3 $ has continuous derivative and thus it's locally {lipschitz} in $\mathbb{R}^{3}$. Hence from fundamental  existence and uniqueness theorem \cite{agarwal2008existence, sibuya1999basic},  we  can conclude the existence of unique solution for the system   (\ref{sec2equ1}) - (\ref{sec2equ3}).
\end{proof} 

\subsection{{\bf{Equilibrium points and the reproduction number (\texorpdfstring{$\mathcal{R}_{0}$}{TEXT})}}} 

The basic reproduction number for the system  (\ref{sec2equ1}) - (\ref{sec2equ3}) is  calculated using the next generation matrix method \cite{heffernan2005perspectives} and the expression for $\mathcal{R}_{0}$ is found to be 
$$\mathcal{R}_{0}=\frac{\alpha \beta \omega}{(\gamma+\mu_{1})(\delta+\mu_{1})(y+\mu_{2})}.$$

 We also see  that the system  (\ref{sec2equ1}) - (\ref{sec2equ3}) admits two equilibria namely, the infection/disease  free equilibrium $E_{0}=\left(\frac{\omega}{\mu_{1}},0,0\right)$ and the infected equilibrium $E^{*}=(S^{*},I^{*},B^{*})$, where
$$S^*=\frac{(\delta+\mu_{1})(y+\mu_{2})}{\alpha \beta}=\frac{\omega}{(\gamma+\mu_{1})\mathcal{R}_{0}}$$
$$I^*=\frac{\alpha \beta \omega-(\gamma+\mu_{1})(\delta +\mu_{1})(y+\mu_{2})}{\alpha \beta (\delta+\mu_{1})}=\frac{(\gamma+\mu_{1})(y+\mu_{2})(\mathcal{R}_{0}-1)}{\alpha\beta}$$

 $$B^*=\frac{\alpha \beta \omega-(\gamma+\mu_{1})(\delta +\mu_{1})(y+\mu_{2})}{ \beta (\delta +\mu_{1}) (y+\mu_{2})}=\frac{(\gamma+\mu_{1})(\mathcal{R}_{0}-1)}{\beta}$$.

\subsection{{\bf{Stability Analysis of \texorpdfstring{$E_{0}$}{TEXT}}}} 

\textbf{Local Stability:}\\

In the following we do the local stability analysis of the infection free equilibrium $E_{0}.$
\vspace{0.01 in}

The {Jacobian} matrix of the system at the infection free equilibrium $E_0$ is given by, 

\begin{equation*}
J_{E_{0}} = 
\begin{pmatrix}
-(\gamma+\mu_{1}) & 0 & \frac{-\beta\omega}{(\gamma+\mu_{1})} \\
0 & -(\delta+\mu_{1}) & \frac{\beta\omega}{(\gamma+\mu_{1})} \\
0 & \alpha & -(y+\mu_{2})
\end{pmatrix}
\end{equation*} 
The characteristic equation is given by, 

\begin{equation}
    \bigg(-(\gamma+\mu_{1})-\lambda\bigg)\bigg[\lambda^2+\{(\gamma+\mu_{1})+(y +\mu_{2})\}\lambda+(\gamma+\mu_{1})(y +\mu_{2})-\frac{\beta \alpha \omega}{(\gamma+\mu_{1})}\bigg]=0
\end{equation}
One of the  eigenvalues of the above equation is  $\lambda_{1}=-(\gamma+\mu_{1})$ which is less then zero  and the other two eigenvalues are calculated as follows:

Introducing $\mathcal{R}_{0}$ in the rest part of the equation 
\begin{equation}
 \lambda^2+\{(\gamma+\mu_{1})+(y +\mu_{2})\}\lambda+(\gamma+\mu_{1})(y +\mu_{2})(1-\mathcal{R}_{0})=0
\end{equation}
Letting  $A  (\gamma+\mu_{1})+(y +\mu_{2}) $ and $ D=(\gamma+\mu_{1})(y +\mu_{2})$ the roots of the above equation are given by
$$\lambda = \frac{1}{2}\big[-A \pm \sqrt{A^2+ 4(\mathcal{R}_{0}-1)D}\big]$$   
We now consider the following two cases for understanding the stability of infection free equilibrium. \\

\textbf{Case I:} {When $\mathcal{R}_{0} < 1$}\\

Further in this case we need to consider the following two sub cases:

\hspace{1cm}\textbf{(a)} $A^2+ 4(\mathcal{R}_{0}-1)D > 0$ \vspace{.3cm}

\hspace{1cm}\textbf{(b)}: $A^2+ 4(\mathcal{R}_{0}-1)D < 0$

\textbf{Sub-case (a)}: When $A^2+ 4(\mathcal{R}_{0}-1)D > 0$ then the eigenvalues are given by, 

$$\lambda_{2,3}=A \pm \sqrt{A^2+ 4(\mathcal{R}_{0}-1)D}  $$  which are less than zero.

Therefore the infection free equilibrium point $E_{0}$ is asymptotically stable in this case as all the eigenvalues are negative.
\hspace{0.1 in}

\textbf{Sub-case (b)}:  When $A^2+ 4(\mathcal{R}_{0}-1)D < 0$ the eigenvalues are complex conjugates with the negative real parts. Therefore in this case also we have $E_{0}$ to be asymptotically stable.
\vspace{0.1 in} 

Hence we conclude that $E_{0}$ is locally asymptotically stable (LAS) whenever $\mathcal{R}_{0} < 1.$ 
\vspace{0.1 in}

\textbf{Case II:} When $\mathcal{R}_{0} > 1$

In this case the characteristic equation has two negative eigenvalues and one positive eigenvalue. Hence whenever $\mathcal{R}_{0}> 1$ the infection free equilibrium $E_{0}$ becomes unstable. \\

\textbf{Global Stability:}\\

As in \textit{Andrei Korobeinikov} \cite{korobeinikov2004global}, we consider the \textit{Lyapunov function} of the system  (\ref{sec2equ1}) - (\ref{sec2equ3})  as
$$U(S,I,B)= S_{0}\bigg(\frac{S}{S_{0}}-\ln{\frac{S}{S_{0}}}\bigg)+ I +\frac{(\delta+\mu_{1})}{\alpha}B$$
Now $$ \frac{dU}{dt}= \omega\bigg(2-\frac{S}{S_{0}}-\frac{S_{0}}{S}\bigg)+\frac{(\delta+\mu_{1})(y+\mu_{2})}{\alpha}(\mathcal{R}_{0}-1)B$$
Here $\bigg(2-\frac{S}{S_{0}}-\frac{S_{0}}{S}\bigg)<0$ and for $\mathcal{R}_{0} < 1,$  the derivative $\frac{du}{dt} < 0$.\\

$\therefore$ For $\mathcal{R}_{0} < 1$ the disease free equilibrium $E_{0}$ is Globally Asymptotically Stable (GAS).

\subsection{{\bf{Stability Analysis of \texorpdfstring{$E^{*}$}{TEXT}}}} 

\textbf{Local Stability:}\\

The {Jacobian} matrix of the system  for $E^*$ is given by
$$
J = 
\begin{pmatrix}
-(\gamma+\mu_{1})\mathcal{R}_{0} & 0 & -\frac{(\delta+\mu_{1})(y+\mu_{2})}{\alpha} \\
(\gamma+\mu_{1})(\mathcal{R}_{0}-1) & -(\delta+\mu_{1}) & \frac{(\delta+\mu_{1})(y+\mu_{2})}{\alpha}\\
0 & \alpha & -(y+\mu_{2})
\end{pmatrix}
$$
The characterstic equation of the {Jacobian} $J$ evaluated at $E^*$ is given by,
$$\lambda^3 + \bigg(p+(\gamma+\mu_{1})\mathcal{R}_{0}\bigg)\lambda^2 + \bigg(p(\gamma+\mu_{1})\mathcal{R}_{0}\bigg)\lambda + q(\gamma+\mu_{1})\bigg(\mathcal{R}_{0}-1\bigg) = 0 \hspace{2cm}$$ 
where $p=(\gamma+\mu_{1})(y+\mu_{2})$ and $q=(\gamma+\mu_{1})(y+\mu_{2}).$
\vspace{0.01 in}

Since $\mathcal{R}_{0} > 1$, $(p+(\gamma+\mu_{1})\mathcal{R}_{0})> 0,(p(\gamma+\mu_{1})\mathcal{R}_{0}) > 0$ and $ q\mu_{1}(\mathcal{R}_{0}-1) > 0 $. Therefore if we substitute $\lambda=-\lambda$ in the above characteristic equation,  we get all the roots of equation to be negative from Descartes rule of sign change. Hence we conclude that the infected equilibrium point $E_{1}$ exists and remains asymptotically stable whenever $\mathcal{R}_{0} > 1$. \\

\textbf{Global Stability:}\\

Considering the \textit{Lyapunov function} of the system  (\ref{sec2equ1}) - (\ref{sec2equ3}) for $E^*$ as in \cite{korobeinikov2004global}
$$U^*(S,I,B)= S^{*}\bigg(\frac{S}{S^{*}}-\ln{\frac{S}{S^{*}}}\bigg)+ I^{*}\bigg(\frac{I}{I^{*}}-\ln{\frac{I}{I^{*}}}\bigg) +\frac{(\delta+\mu_{1})}{\alpha}B^{*}\bigg(\frac{B}{B^{*}}-\ln{\frac{B}{B^{*}}}\bigg)$$
 we can show the GAS of  $E^{*}$ when $\mathcal{R}_{0 }> 1.$ \\
 
 \subsection{{\bf{Bifurcation Analysis} }} \vspace{.2cm}
 
 We now use the method given by \textit{Bruno Buonomo} in \cite{buonomo2015note} to do the bifurcation analysis for the system  (\ref{sec2equ1}) - (\ref{sec2equ3}). \vspace{.2cm}

 \begin{theorem}
 The system   (\ref{sec2equ1}) - (\ref{sec2equ3}) undergoes a trans-critical bifurcation at $\mathcal{R}_{0}=1$ and it is forward.  
 \end{theorem}
 \begin{proof}
 Let's consider $x_{1}=I, x_{2}=B, x_{3}=S$ and  $x=(x_{1},x_{2},x_{3}).$ \\
 
 Now 

 \[   
Infected\_class 
     \begin{cases}
       \dfrac{dx_{1}}{dt} & = \beta x_{2}x_{3}-(\delta+ \mu_{1})x_{1}\\
       
       \dfrac{dx_{2}}{dt} & = \alpha x_{1}-(y+ \mu_{2})x_{2}\\
     \end{cases}
\]
\[   
Uninfected\_class 
     \begin{cases}
       \dfrac{dx_{3}}{dt} & = \omega-\beta x_{2}x_{3}-(\gamma+ \mu_{1})x_{3}\\
       
     \end{cases}
\]

We consider $f(x)=(f_{1},f_{2},f_{3})=\bigg(\dfrac{dx_{1}}{dt},\dfrac{dx_{2}}{dt},\dfrac{dx_{3}}{dt}\bigg)$, hence $f(x)$ is twice differentiable function in $\mathbb{R}^{3}.$

Further we can interpret each $f_{i}$ as

$$ f_{i}(x)= \mathcal{F}_{i}(x)-\mathcal{V}_{i}(x), \ i = 1, 2, 3$$
Where $\mathcal{V}_{i}={V}_{i}^{-}-{V}_{i}^{+}$ and here 
\begin{itemize}
    \item $\mathcal{F}_{i}:=$Appearance rate of new infection in $i^{th}$ compartment 
    \item $\mathcal{V}_{i}^{+}:=$ Transfer rate of individuals \textbf{into} the $i^{th}$ compartment.
    \item $\mathcal{V}_{i}^{-}:=$ Transfer rate of individuals \textbf{out of} the $i^{th}$ compartment.
\end{itemize}
Therefore here
\begin{itemize}
    \item $\mathcal{F}_{1}= \beta x_{2}x_{3} ,\mathcal{V}_{1}^{+}=0, \mathcal{V}_{1}^{-}=(\delta + \mu_{1})x_{1}$
    \item $\mathcal{F}_{2}= \alpha x_{2} ,\mathcal{V}_{2}^{+}=0, \mathcal{V}_{2}^{-}=(y + \mu_{2})x_{2}$
    \item $\mathcal{F}_{3}= 0 ,\mathcal{V}_{3}^{+}=0, \mathcal{V}_{3}^{-}=\beta x_{2}x_{3}+(\gamma + \mu_{1})x_{3}$
\end{itemize}

Denote $\mathcal{X}_{s}$ as the set of all disease free state i.e.
$$\mathcal{X}_{s}:= \{x\in \mathbb{R}^{3}: x_{1}=0,x_{2}=0\}=\bigg\{\bigg(0,0,\frac{\omega}{(\gamma+\mu_{1})}\bigg)\bigg\}$$

Now we will satisfy the condition \textbf{A1 - A5} of \cite{buonomo2015note} as follows \vspace{0.1cm}

\textbf{A1:} All $\mathcal{F}_{i},\mathcal{V}_{i}^{+}$ and $\mathcal{V}_{i}^{-}$ are positive for $i=1,2,3$ in the nonnegative cone $\{x\in\mathbb{R}:x_{i}\geq 0,i=1,2,3\}$\\
\textbf{A2:} If $x\in \mathcal{X}_{s}$ then $\mathcal{V}_{i}^{-}=0$ for the infected compartment, i.e. $i=1,2$. Since for $x\in \mathcal{X}_{s}$ we have $x_{1}=0$ and $x_{2}=0$
$$\implies\mathcal{V}_{1}^{-} = (\delta+\mu_{1}) . 0 = 0 , \mathcal{V}_{2}^{-} = (y+\mu_{2}) . 0 = 0  $$\\
\textbf{A3:} No incidence of infection in uninfected compartment($x_{3}$), that is  $\mathcal{F}_{3}= 0$\\
\textbf{A4:} Disease free subspace is invariant, that means for $x\in \mathcal{X}_{s}$, here $\mathcal{F}_{i}=0,\mathcal{V}_{i}^{+}=0$, $i=1,2$\\
\textbf{A5:} Now putting all $\mathcal{F}_{i}=0$, we have 
$$f(x)=\big(-(\delta+ \mu_{1})x_{1}, -(y+ \mu_{2})x_{2},  \omega-\beta x_{2}x_{3}-(\gamma+ \mu_{1})x_{3}\big)$$
 Now the derivative matrix of $f(x)$ is given by
\begin{eqnarray*}
       \mathcal{D}_{f(x)} &=& \begin{bmatrix}
    -(\delta+\mu_{1}) & 0 & 0  \\
     0 & -(y+\mu_{2}) & 0\\
     0 & \beta x_{3}  & -\beta x_{2} - (\gamma+\mu_{1})
\end{bmatrix}\\
\implies \mathcal{D}_{f(x_{0})} &=& \begin{bmatrix}
-(\delta+\mu_{1}) & 0 & 0  \\
     0 & -(y+\mu_{2}) & 0\\
     0 & \frac{-\beta \omega}{(\gamma + \mu_{1})}  & - (\gamma+\mu_{1})
\end{bmatrix}
\end{eqnarray*}
where $x_{0}\in \mathcal{X}_{s}$, and here $\mathcal{D}_{f(x_{0})}$ is a lower triangular matrix with all negative diagonal entries and hence all the {eigen values} illustrating  that the {disease free equilibrium} is stable in the absence of new infections.
\vspace{0.5cm}

We now show that the following hypothesis \textbf{H1 - H3} of \cite{buonomo2015note}, is also satisfied. \\

\textbf{H1:} The only nonlinear term present in infected compartment of the system is $\mathcal{F}_{1} = \beta  x_{2}x_{3}$\\
\textbf{H2:} Let $T(x_{2},x_{3})=\beta  x_{2}x_{3}$
\begin{itemize}
    \item  $T(k x_{2},x_{3})=\beta k .x_{2}x_{3}= k. \beta  x_{2}x_{3}= k .T(x_{2},x_{3})$ 
    \item $T(x_{2},k x_{3}) = \beta  x_{2}.k x_{3} = k. \beta  x_{2}x_{3}= k .T(x_{2},x_{3}) $
    \item $T(x_{2}+x'_{2},x_{3}) = \beta (x_{2}+x'_{2})x_{3}=\beta x_{2}x_{3}+ \beta x'_{2}x_{3} = T(x_{2},x_{3}) + T(x'_{2},x_{3})$
    \item $T(x_{2},x_{3}+x'_{3}) = \beta x_{2}(x_{3}+x'_{3})=\beta x_{2}x_{3}+ \beta x_{2}x'_{3} = T(x_{2},x_{3}) + T(x_{2},x'_{3})$
\end{itemize}
$\therefore$ The nonlinear term in the above hypothesis (\textbf{H1}) is bilinear in nature.
\vspace{0.1 cm}

\textbf{H3:} There  is no transfer from infected compartment to uninfected compartment. \\

Now using \textit{Proposition-1} in the paper \cite{buonomo2015note} we can conclude that  the system  (\ref{sec2equ1}) - (\ref{sec2equ3})  undergoes a trans-critical bifurcation at  $\mathcal{R}_{0} = 1$  which is forward in nature.
\end{proof}

 \section{Numerical Simulations} \vspace{.2cm}
 
 All the values of the parameters used here are estimated  from different clinical papers. The appropriate references are cited in the Table 1. Some parameters are minimally fine tuned from the Table 1 values to satisfy certain hypothesis assumptions in some of the following plots.
 
 
 \begin{table}[ht!]   
 \centering 
\begin{tabular}{|c|c|c|} 
\hline

\textbf{Symbols} &  \textbf{Values} & \textbf{Units} \\  

\hline\hline
$\omega$ & 0.022 \cite{kim2017schwann} & $day^-1$ \\

\hline\hline
$\beta$ & 3.44 \cite{jin2017formation} & $day^{-1}$  \\

\hline\hline
$\gamma$& 0.1795 \cite{oliveira2005cytokines} & $day^{-1}$ \\

\hline\hline
$\mu_{1}$ & 0.0018 \cite{oliveira2005cytokines} & $day^{-1}$ \\

\hline\hline
$\delta$ & 0.2681 \cite{oliveira2005cytokines} & $day^{-1}$\\

\hline\hline
$\alpha$ & 0.063 \cite{levy2006mouse} & $day^{-1}$ \\

\hline\hline
$y$ & $0.0003$ \cite{ghosh2021mathematical}& $day^{-1}$   \\
\hline\hline
$\mu_{2}$ & 0.57 \cite{international2020international}& $day^{-1}$\\
\hline

\end{tabular}
\caption{Values of the parameters complied from clinical literature }
\label{Table:1}
\end{table} 

\subsection{{\bf{Disease free equilibrium  \texorpdfstring{$E_{0}$}{TEXT}}}}

\hspace{0.5 in}
\hspace{0.5 in} We now depict the local and global stability of the disease free equilibrium $E_{0}.$ Figures 1a and 1b depict the local and global stability of $E_{0}.$ 

 \hspace{0.5 in} We choose parameters  in Table \ref{Table:2} in such a way that $\mathcal{R}_{0}=0.9939 < 1$ and for these parameters we have  $E_{0} = (55.1899,0,0)$. For depicting the global stability of  $E_{0}$ we have arbitrarily considered the solution trajectories taking ten different initial conditions.  
 
\begin{table}[ht!]   
 \centering 
\begin{tabular}{|c|c|c|c|c|c|c|c|} 
\hline

 $\omega$ & $\beta$ & $\gamma$ & $\mu_{1}$ & $\delta$ & $\alpha$ & $y$ & $\mu_{2}$ \\
\hline\hline
  1.090 & 0.44 & 0.01795 & 0.0018 & 0.2681 & 0.0063 & 0.0003 & 0.57 \\
 \hline
\end{tabular}
\caption{Values of the parameters taken for $E_{0}$ }
\label{Table:2}
\end{table}

\begin{figure} 
    \begin{subfigure}{0.45\textwidth}
        \includegraphics[width=\textwidth]{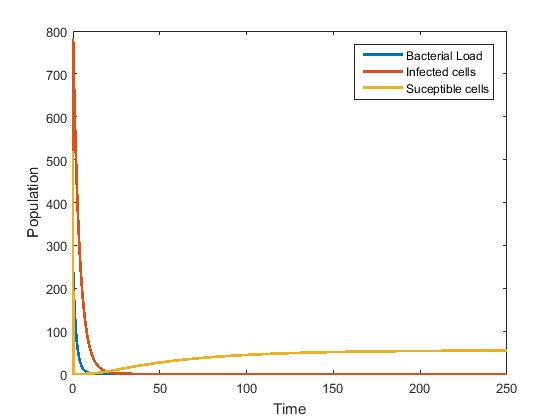}
        \caption{ }
        \label{1a}
    \end{subfigure}
    \begin{subfigure}{0.45\textwidth}
        \includegraphics[width=\textwidth]{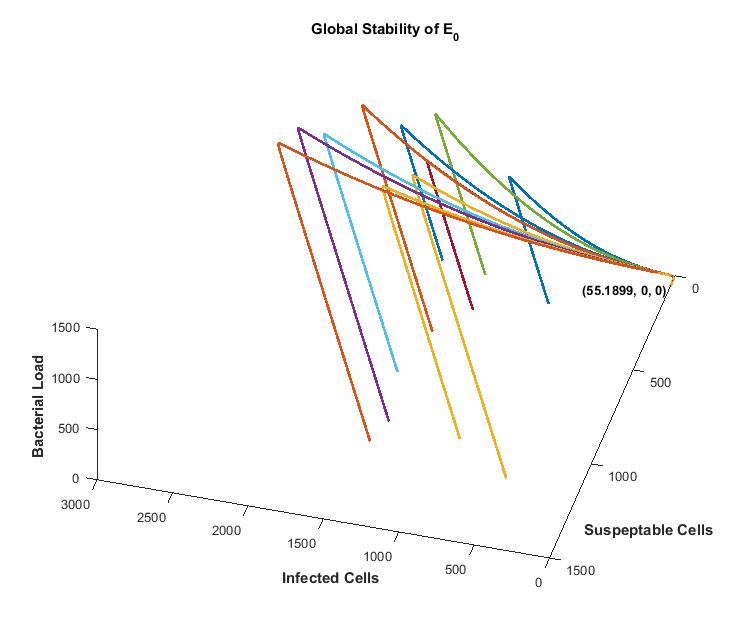}
        \caption{ }
        \label{1b}
    \end{subfigure}
    
    \caption{Local and global Stability of the system (\ref{sec2equ1}) - (\ref{sec2equ3})  at $E_{0}$ }
    \label{E0}
\end{figure}

\subsection{{\bf{ Infected/endemic equilibrium \texorpdfstring{$E^{*}$}{TEXT}}}} 

\hspace{0.5 in} We now depict the local and global stability of the endemic equilibrium $E^{*}.$ Figures 2a and 2b depict the local and global stability of $E^{*}.$

\hspace{0.5 in}  For the  numerical simulations we have chosen the values of parameters as in   Table \ref{Table:3}.  For these parameter values we have $\mathcal{R}_{0}  = 29.6341 > 1$ and the $E^{*} = (38.9006, 75.2748, 17.3046)$.  For depicting the global stability of  $E^{*}$ we have arbitrarily considered  solution  trajectories with  different initial conditions.

\begin{table}[ht!]   
 \centering 
\begin{tabular}{|c|c|c|c|c|c|c|c|} 
\hline

 $\omega$ & $\beta$ & $\gamma$ & $\mu_{1}$ & $\delta$ & $\alpha$ & $y$ & $\mu_{2}$ \\
\hline\hline
  20.90 & 0.030 & 0.01795 & 0.00018 & 0.2681 & 00.2 & 0.3 & 0.57 \\
 \hline
\end{tabular}
\caption{Values of the parameters taken for $E^{*}$ }
\label{Table:3}
\end{table} 

\begin{figure} 
    \begin{subfigure}{0.45\textwidth}
        \includegraphics[width=\textwidth]{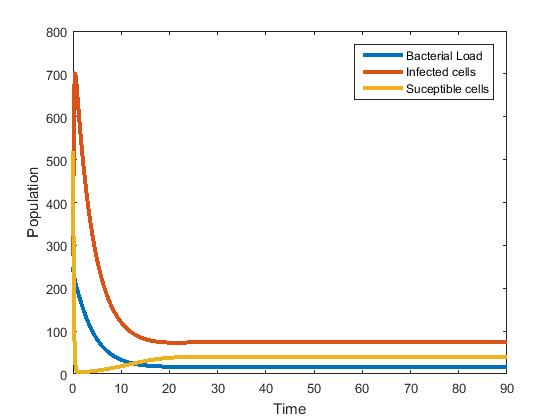}
        \caption{ }
        \label{2a}
    \end{subfigure}
    \begin{subfigure}{0.45\textwidth}
        \includegraphics[width=\textwidth]{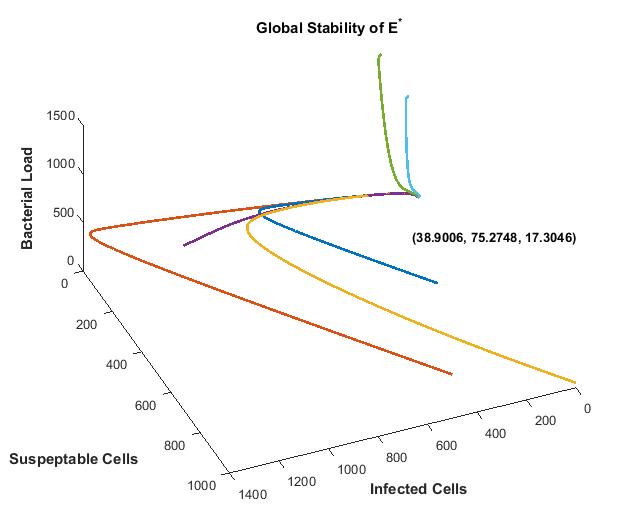}
        \caption{ }
        \label{2b}
    \end{subfigure}
    
    \caption{Local and global stability of the system (\ref{sec2equ1}) - (\ref{sec2equ3})  at $E^*$.}
    \label{fig:2}
\end{figure}

\subsection{{\bf{Transcritical Bifurcation}}}

\hspace{0.5 in} In this bifurcation there is an exchange of stability between $E_{0}$ and $E^{*}$ as $\mathcal{R}_{0}$ crosses unity. To  depict this bifurcation, we varied the parameter $\omega$ from $0$ to $0.25$ with step size $0.001$ and chose the other parameters from Table \ref{Table:1}.  The Figure  \ref{tcbf} depicts the occurrence of Transcritical bifurcation at $\mathcal{R}_{0} = 1$ . 

\begin{figure}[htt!]
    \centering
    \includegraphics[height = 6cm, width =7.5cm]{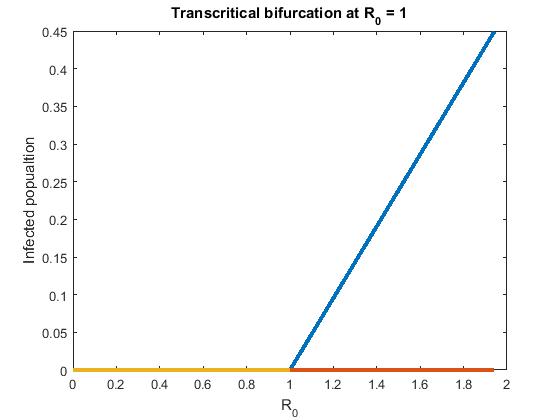}
    \caption{Figure depicting the transcritical bifurcation exhibited by the system (\ref{sec2equ1}) - (\ref{sec2equ3})  at  $\mathcal{R}_{0} = 1$}
    \label{tcbf}
\end{figure} \vspace{.2cm}

\newpage

\section{Model validation through 2D Heat Plots}  \vspace{.2cm}

\hspace{0.5 in}  Form some of the clinical studies we see that the average doubling time of the \textit{M. Leprae} is approximately \textit{14 days} \cite{pinheiro2011mycobacterium}. Based on this  characteristic, we now validate the model (\ref{sec2equ1}) - (\ref{sec2equ3}) through 2D heat plot. 

\hspace{0.5 in} We now vary the parameters $\alpha$ from $0.2263$ to $0.3099$ on the $x - axis$ and the parameter $\gamma$ between  $0.15$ to $0.2090$ on the $y - axis$ and generate a two parameter heat plot to validate our model (\ref{sec2equ1}) - (\ref{sec2equ3}). All other parameter are taken from Table \ref{Table:1} and the initial condition was chosen to be $(S_{0},I_{0},B_{0}) = (5200,0,40)$.

\begin{figure}[htt!]
    \centering
    \includegraphics[height = 7cm, width =9cm]{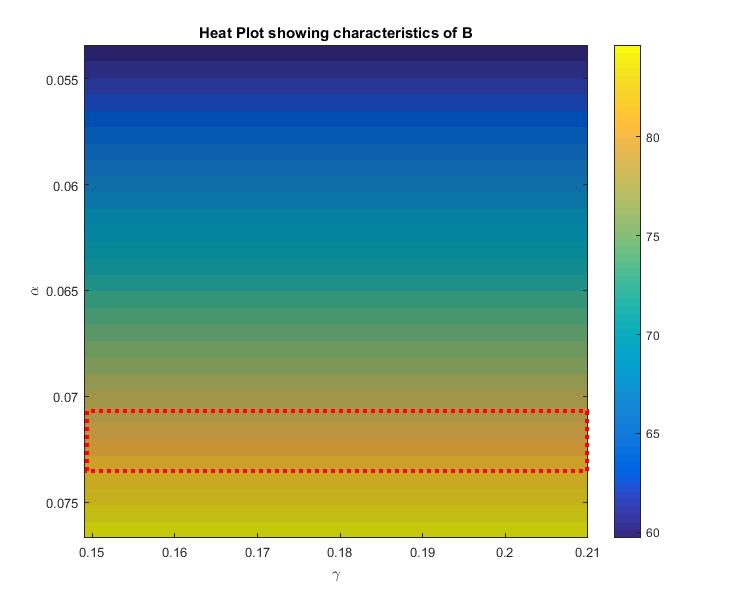}
    \caption{}
    \label{HP}
\end{figure}
\hspace{0.5 in} Now from the Figure \ref{HP} it can be  seen that the proposed model is able to reproduce characteristic, i.e. exactly the double of initial count of bacterial load that is $80$ ($B_{0} = 40$), indicated by the dotted red rectangle. \\

\section{Sensitivity Analysis} \vspace{.2cm}

\hspace{0.5 in} Here we are interested in investigating the impact of uncertainty in the values of the different parameters on the variables ($S, I, B$). For this we have used the Global Sensitivity Analysis (GSA) methodology through {Latin hyper cube sampling} (LHS). LHS is a technique that involves
sampling without replacement a set of model parameter combinations from preset ranges on the parameter values \cite{cho2003experimental, marino2008methodology, zhang2015sobol}. Using this sample we generate the scatter plot to decide the methodology for GSA.  The scatter plots enables the graphical detection of the non-linearities, non-monotonicities  between model input (parameter) and output (variables). If the trend is non-linear then rank correlation coefficient such as  {Partial Rank Correlation Coefficient} (PRCC) , {Spearman's Rank Correlation Coefficient} (SRCC) will be used for further sensitivity analysis where as if the trend is non-monotonic, method based on decomposition of model output variance such as  {Sobol's method} will be the best choice for   further analysis.

\subsection{{\bf{LHS and Scatter plots} }} \vspace{.2cm}

\hspace{0.5 in} As an initial step to LHS we select the following parameters listed in Table \ref{lhs} having possible uncertainty in their values and consider them for the process of sampling. The range of the variable values used for sampling is listed in Table \ref{lhs}. All the parameter value ranges are chosen based on the clinical papers  \cite{ levy2006mouse, oliveira2005cytokines} and we introduced an uncertainty in $y$ for our computational convenience. The remaining values of the parameters are as in Table  \ref{Table:1}. \\

\begin{table}[ht!]   
\centering 
\begin{tabular}{|c|c|c|} 
\hline

\textbf{Parameter} &  \textbf{Max Value} & \textbf{Min Value} \\ 

\hline\hline
$\gamma$& 0.0763 & 0.0538 \\

\hline\hline
$\mu_{1}$ & 0.0405 & 0.0305 \\

\hline\hline
$\delta$ & 0.3099 & 0.2263\\

\hline\hline
$\alpha$ & 0.0763 & 0.0538 \\

\hline\hline
$y$ & 0.0001 & 0.0005   \\
\hline

\end{tabular}
\caption{Range of sensitive parameters }
\label{lhs}
\end{table}

\hspace{0.5 in}Then the LHS is done to create $1000$ sets of parameter sample each containing $5$ random values of parameters. Now each set of these parameters was used to simulate the model at each time. Scatter plots were created for each parameter vs variable to decide the further procedure of GSA.        
\begin{figure}[ht!]
    \centering
    \includegraphics[height = 7cm, width =16cm]{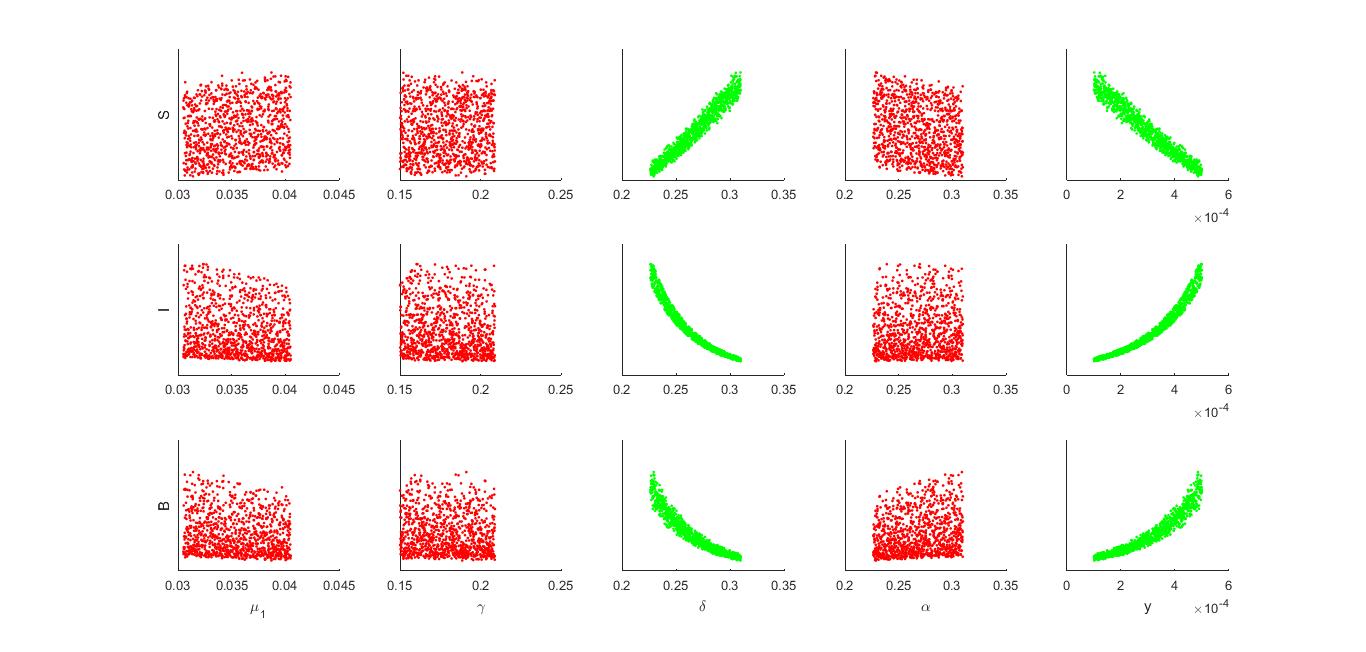}
    \caption{Scatter plots for  parameters vs variables such as $S(t),   I(t) \  \text{and}  \ B(t)$  }
    \label{SP}
\end{figure}

\hspace{0.5 in} In the {Figure} \ref{SP} we can easily see that the relationships between $\delta$ and all variables such as $S(t),   I(t) \  \text{and}  \ B(t)$ follow a monotonic trend and the so is the case for $y$. Therefore we did the SRCC and PRCC for these two variable and the remaining parameters were analysed by calculating the Sobol's index. \\

\subsection{{\bf{SRCC and PRCC}}}

 \hspace{0.5 in} Using the same sample obtained above, we calculated SRCC index   separately for $\delta $ and $y$ and  the PRCC index jointly.
 
 \begin{figure}[htt!]
    \centering
    \includegraphics[height = 8cm, width =16cm]{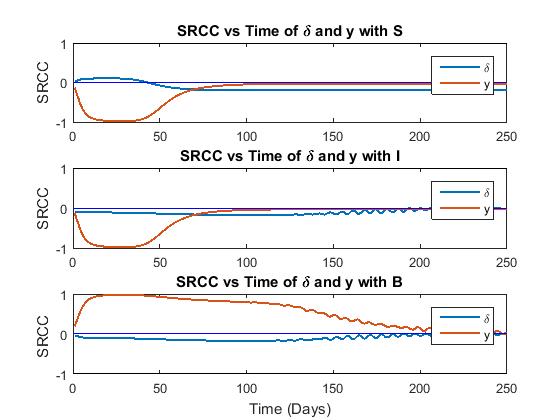}
    \caption{Plot for SRCC with respect to time  }
    \label{srcc}
\end{figure}
\begin{figure}[htt!]
    \centering
    \includegraphics[height = 7cm, width =16cm]{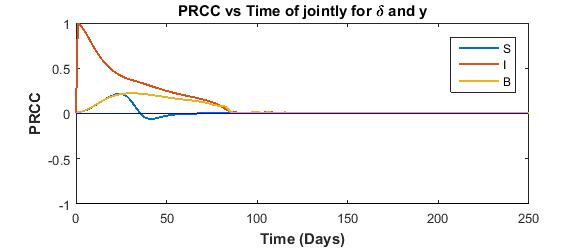}
    \caption{Plot for PRCC with respect to time  }
    \label{prcc}
\end{figure}

\newpage
\hspace{0.5 in}The {Figure} \ref{srcc} shows that $\delta$ has more negative impact on $S$ and $I$ in comparison to  $y$, whereas $y$ has more positive impact on $B$. The $PRCC$ plot  \ref{prcc} shows that the cummulative impact of $\delta$ and $y$ seems to be more on the Infected cell population $I.$

\newpage

\subsection{{\bf{Sobol's Index}}} \vspace{.2cm}

\hspace{0.5 in} The Sobol's index is caluculated using the formula of correlation \cite{saltelli2008global}.
$$S_{i}=Corr(Y,E(Y/x_{i}))$$
where $S_{i}$ is the Sobol's index of $i^{th}$ parameter, $Y$ is the model out put value and  $E(Y/x_{i})$ conditional expectation/ mean of model output Y.\\

\hspace{0.5in} The Sobol's index was calculated for the parameters $\mu_1, \gamma, \alpha$ at each time as in SRCC and was plotted separately for the model variables $S, I and B$ which can be seen in Figure \ref{sbl}. Unfortunately from these plots  any we couldn't derive fruitful conclusions to decide the most sensitive parameter owing to the high fluctuations.  \\

\hspace{0.5in} Because of the above limitation we tried to identify the sensitive parameters with respect to $\mathcal{R}_{0}$ which is discussed in next section. \\


\begin{figure} 
\centering
    \begin{subfigure}{0.49\textwidth}
        \includegraphics[width=\textwidth]{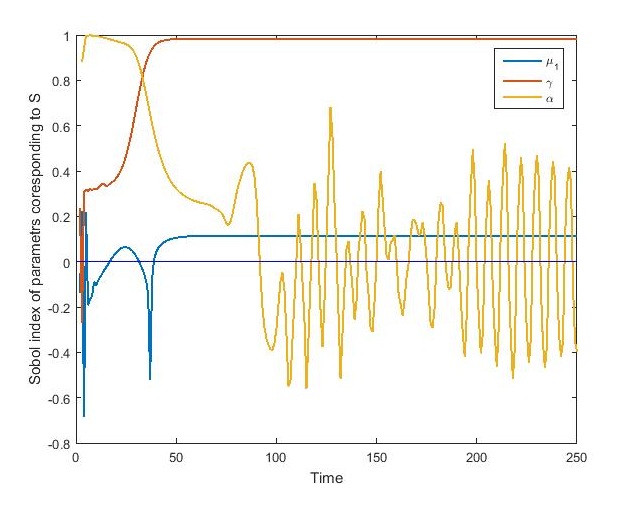}
        \caption{ }
        \label{4a}
    \end{subfigure}

    \begin{subfigure}{0.5\linewidth}
        \includegraphics[width=\textwidth]{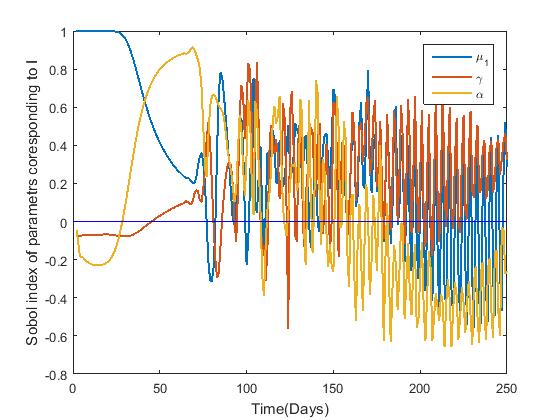}
        \caption{ }
        \label{4b}
    \end{subfigure}

    \begin{subfigure}{0.5\linewidth}
        \includegraphics[width=\textwidth]{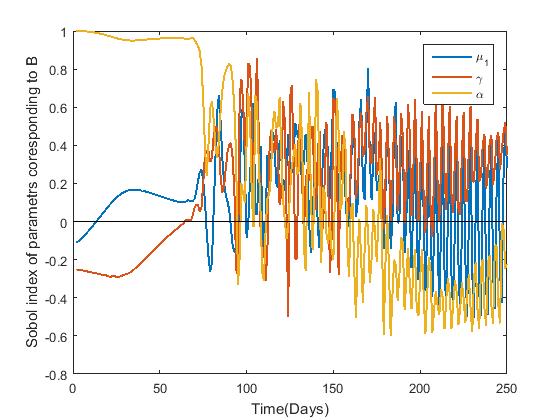}
        \caption{ }
        \label{4c}
    \end{subfigure}
    
    \caption{Plot of Sobol's index  of $\mu_{1},$ $\gamma$ and $\alpha$ with $S, I$ and $B$ respectively}
    \label{sbl}
\end{figure}

\subsection{{\bf{Sensitivity of \texorpdfstring{$\mathcal{R}_{0}$}{TEXT}} }} \vspace{.2cm}

\hspace{0.5 in} For identifying the sensitive parameters with respect to ${\mathcal{R}_{0}},$ we  did the scatter plots of the parameters against $\mathcal{R}_{0}$ and saw that none of them were qualified for PRCC analysis. Hence we calculated the {Sobol's sensitivity index} for each parameter and pairs of parameters as listed in frames of the plot \ref{sblR0}.

\begin{figure} 
    \begin{subfigure}{0.45\textwidth}
        \includegraphics[width=\textwidth]{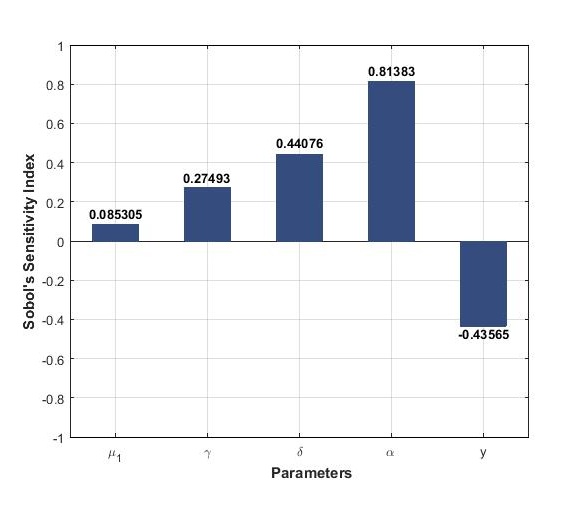}
        \caption{ }
        \label{9a}
    \end{subfigure}
    \begin{subfigure}{0.45\textwidth}
         \includegraphics[width=\textwidth]{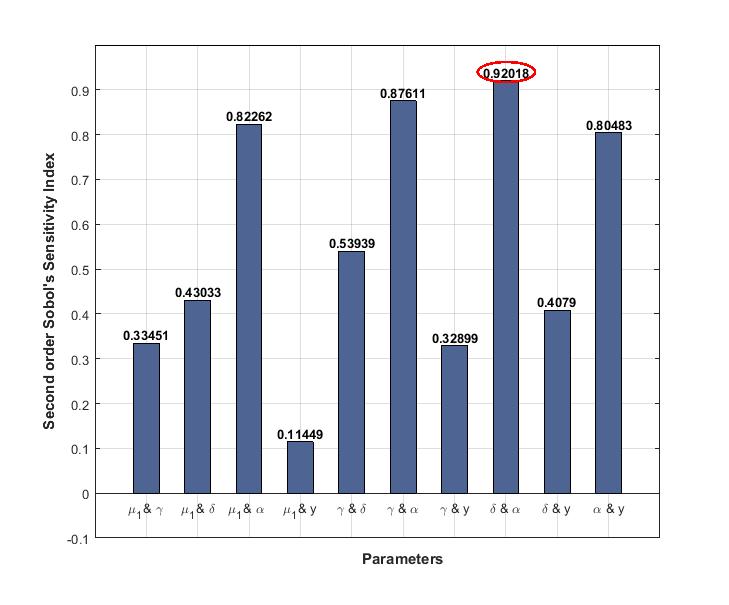}
        \caption{ }
        \label{9b}
    \end{subfigure}
    
    \caption{Sensitivity through Sobol's Index with respect to  $\mathcal{R}_{0}$ as  output}
    \label{sblR0}
\end{figure}

\subsection{{\bf{Inference} }}
\hspace{0.5 in} From the above sensitivity analysis  we can conclude that $\alpha$ is the most sensitive parameter followed by $\delta$ and $y$. Here $\alpha$ and $\delta$ have direct impact on the system where as $y$ has inverse impact on the system as it has a negative Sobol's index. In case of cumulative parameter sensitivity we see  that the parameter combination of $\alpha$ and $\delta$ is the most sensitive combination that impacts the system (\ref{sec2equ1}) - (\ref{sec2equ3}).

\newpage

\section{Optimal Control Studies} \vspace{.2cm}

     \hspace{0.5 in} Presently for the type - I  lepra reaction  two  kinds of medication are prescribed based on the disease condition  \cite{maymone2020leprosy, walker2008leprosy}.  
     Firstly {Multi Drug Therapy} (MDT) is used and in case still the reaction burden doesn't reduce, then  steroids are given along with MDT treatment.

        \hspace{0.45 in} Motivated by the above clinical findings in this section we frame and study two optimal control problems. First one deals with the optimal drug regimen for MDT and the second deals with the optimal drug regimen for the scenario involving both MDT and steroid interventions.  These medical/drug interventions are modeled as control variables for the system (\ref{sec2equ1}) - (\ref{sec2equ3}).
      
\subsection{{\bf{Optimal control problem associated with MDT }}} \vspace{.2cm}

    \hspace{0.5 in} According to the WHO recommended guidelines of 2018 for Leprosy MDT consist of three drugs \textit{Rifampin, Dapsone }and \textit{Clofazimine} \cite{maymone2020leprosy, tripathi2013essentials}. The drug {rifampin} acts  as a  rapid bacillary killer and thereby indirectly reduces the amount of cells getting infected. Therefore the control variable $D_{12}(t)$ is negatively incorporated   in the infected cell compartment of (\ref{sec11equ1}) -  (\ref{sec11equ3}) and $D_{13}^{2}(t)$ is negatively incorporated  in the bacterial load compartment of (\ref{sec11equ1}) -  (\ref{sec11equ3}). Here the square on $D_{13}(t)$ is used for capturing the extent of intense action of this drug on bacterial load. The drug {dapsone} is bactericidal and
     bacteriostatic against \textit{M. leprae} and it also has some adverse effect of nerve damage due to the cytokines responses \cite{paniker2001dapsone}. To capture this action of the drug  we incorporate  $D_{21}(t)$ and $D_{22}(t)$  in the compartments $S$ and $I$ of (\ref{sec11equ1}) -  (\ref{sec11equ3}) and  $D_{23}^2(t)$ in the $B$ compartment of (\ref{sec11equ1}) -  (\ref{sec11equ3}). The third drug {clofazimine} has an immuno-suppressive effect and also it binds with DNA of the bacteria causing the inhibition of template function of DNA resulting bacteriostatic against \textit{M. leprae } \cite{garrelts1991clofazimine}. To incorporate this phenomenon we add the control variable $D_{31}(t)$ to the $S$ compartment in (\ref{sec11equ1}) -  (\ref{sec11equ3}) resulting increase of these cells. $D_{33}(t)$ is negatively incorporated in the $B$ compartment of (\ref{sec11equ1}) -  (\ref{sec11equ3}) to indicate the inhibition of bacterial replication. 
     
     \hspace{0.5 in} Now mathematically  we define the set of all  control variables as follows:
     
     $$U=\Big\{D_{ij}(t),D_{ij}(t)\in[0,D_{ij}max], 1\leq i,j \leq 3,\ ij \neq 32, \ t\in[0,T]\Big\}$$
     
    \hspace{0.5 in} Here $D_{ij}max$ represents the maximum value of the corresponding  control variable which depends on the availability and limit of the drugs recommended for patients   and $T$ is the final time of observation.
     
     \hspace{0.5 in}  Since the drugs used  in MDT can be toxic and can lead to side effects for solving this optimal control problem we consider a cost functional that minimizes the drug concentrations along with the infected cell count and bacterial load. Based on this we consider the following cost functional:
     
     \begin{equation}
         \mathcal{J}_{min}\big(D_{1},D_{2},D_{3}\big)= \int_{0}^{T} \Big(I(t) + B(t)+P[D^{2}_{11}(t)+D^{2}_{12}(t)+D^{3}_{13}(t)]+Q[|D_{2}|^{2}] + R[|D_{3}|^{2}]\Big) dt 
         \label{opti}
     \end{equation}
            
  subject to the constraints/system 
  
    \begin{eqnarray}
   	\frac{dS}{dt}& =&  \omega \ - \beta SB  - \gamma S - \mu_{1} S -D_{11}(t)S - D_{21}(t)S+D_{31}(t)S\label{sec11equ1} \\
   	\frac{dI}{dt} &=& \beta SB \ -\delta I - \mu_{1} I-D_{12}(t)I - D_{22}(t)I  \label{sec11equ2}\\ 
   	\frac{dB}{dt} &=&  \big(\alpha-D_{23}^2(t)-D_{33}(t)\big) I  \ - y B    -  \mu_{2} B-D_{13}^2(t)B \label{sec11equ3}
   \end{eqnarray} 
   
      Here $D_{1}=\big(D_{11},D_{12},D_{13} \big)$, $D_{2}=\big(D_{21},D_{22},D_{23} \big)$, $D_{3}=\big(D_{31},D_{33} \big)$ and $|\bullet |$ represents standard \textit{Euclidean} norm in $\mathbb{R}^{n}$ and $(D_{1},D_{2},D_{3})\in U.$
   
The integrand  of the cost function \ref{opti}, denoted by 
\begin{equation}
    L\big(I,V, D_1,D_2,D_3\big) = \Big(I(t) + B(t)+P[D^{2}_{11}(t)+D^{2}_{12}(t)+D^{3}_{13}(t)]+Q[|D_{2}|^{2}] + R[|D_{3}|^{2}]\Big)
\end{equation}
 is the {lagrangian} or {running cost} of the optimal control problem. \\
 
The admissible set of solutions for the above optimal control problem (\ref{opti}) - (\ref{sec11equ3}) is given by
$$\Omega =\Big\{\big(I,V,D_1,D_2,D_3\big): I, V satisfying  \hspace{0.1in} (\ref{sec11equ1})-(\ref{sec11equ3})  \ \forall \ (D_1,D_2,D_3) \in U \Big\}$$\\


{\subsection{\bf{Existence of optimal solution}}} \vspace{.2cm}

\hspace{0.5 in} In the section  we   establish the existence of optimal control for the system (\ref{opti}) - (\ref{sec11equ3}) using the existence theorem 2.2 of \cite{boyarsky1976existence} dealing with nonlinear control systems.

\begin{theorem}
There exists a $8$- tuple of optimal controls  $\big(D_{1}^*(t),D_{2}^*(t),D_{3}^*(t)\big)$ in the set of admissible controls $U$ such that the cost function is minimized i.e.
$$ \mathcal{J}\big(D_{1}^*(t),D_{2}^*(t),D_{3}^*(t)\big)=\min_{D_1,D_2,D_3\in U} \big\{ \mathcal{J} \big(D_{1},D_{2},D_{3}\big)\big\}$$
corresponding to the control system  (\ref{opti}) - (\ref{sec11equ3}), where $D_{1}=\big(D_{11},D_{12},D_{13} \big)$, $D_{2}=\big(D_{21},D_{22},D_{23} \big)$, $D_{3}=\big(D_{31},D_{33} \big)$ .
\end{theorem}
\begin{proof}
Let's consider that $\frac{dS}{dt}=f^{1}(t,x,D)$, $\frac{dI}{dt}=f^{2}(t,x,D)$ and
$\frac{dB}{dt}=f^{3}(t,x,D)$ of the control system  (\ref{opti}) - (\ref{sec11equ3}). Here $x\in
X$ denotes the state variables $(S, I, B)$ and $D$ denote $8$-tuple control variables. We take 
$f=\big(f^1,f^2,f^3\big)$, then clearly $X\subset \mathbb{R}^3$ and
$$f:[0,T]\times X \times U \to  \mathbb{R}^3$$
is a continuous function of $t$ and $x$ for each $D_{ij} \in U$. Now  we have to show $(F1)-(F3)$
of  \textit{Theorem 2.2} of \cite{boyarsky1976existence} hold true.\\

 \textbf{F1:} Here each $f^i$'s have the continuous and bounded partial derivatives which imply that the $f$ is  Lipschitz's continuous.\\
 
 \textbf{F2:} We  consider $g_1(D_{11},D_{21},D_{31}) =-D_{11}-D_{12}+D_{31} $ , which is bounded on $U$. Thus 
 \begin{equation*}
  \begin{aligned}
  \frac{ f^1(t,x,D^{(1)})- f^1(t,x,D^{(2)})}{\big[g_1(D^{(1)})-g_1(D^{(2)})\big]}
 &= \frac{\big[D^{(2)}_{11}+D^{(2)}_{12}-D^{(2)}_{31}-D^{(1)}_{11}-D^{(1)}_{12}+D^{(1)}_{31}\big]S}{\big[D^{(2)}_{11}+D^{(2)}_{12}-D^{(2)}_{31}-D^{(1)}_{11}-D^{(1)}_{12}+D^{(1)}_{31}\big]}\\
   & \leq \eta S = F_1(t,x)\\
 \therefore  f^1(t,x,D^{(1)})- f^1(t,x,D^{(2)}) &\leq F_1(t,x)\bullet \big[g_1(D^{(1)})-g_1(D^{(2)})\big]
  \end{aligned}
\end{equation*}
 Here $\eta > 1$ is a real number. Moreover since $U$ compact and $g_1$ is continuous  we have $g_{1}(U)$ to be compact. Also since the function  $g_{1}(U)$ is linear so the range of $g_{1}$ i.e. $g_{1}(U)$ will be convex.  Since $U$ is non-negative so $g^{-1}_{1}$ is non-negative. \\

 Similarly for $f^{2}(t,x,D)$ we can choose $g_2(D_{12},D_{22}) = -D_{12}-D_{22}$ and $F_2(t,x)= I$ and prove  F2  in a similar way.\\

 Now for $f^{3}(t,x,D)$ we have to choose $g_3(D_{23},D_{33}) = -D^{2}_{23}-D_{33} $ 
 
\begin{equation*}
  \begin{aligned}
  \frac{ f^2(t,x,D^{(1)})- f^2(t,x,D^{(2)})}{\big[g_2(D^{(1)})-g_2(D^{(2)})\big]}
   &= \frac{\big[D^{2(2)}_{23}+D^{(2)}_{33}-D^{2(1)}_{23}-D^{(1)}_{33}\big]I-[D_{13}^{2(1)}-D_{13}^{2(2)}]B}{\big[D^{2(2)}_{23}+D^{(2)}_{33}-D^{2(1)}_{23}-D^{(1)}_{33}\big]}\\
   & \leq \frac{\big[D^{2(2)}_{23}+D^{(2)}_{33}-D^{2(1)}_{23}-D^{(1)}_{33}\big]I}{\big[D^{2(2)}_{23}+D^{(2)}_{33}-D^{2(1)}_{23}-D^{(1)}_{33}\big]}=I=F_3(t,x)  \ \bigg(\text{provided} \ D^{2(1)}_{13} \geq D^{2(2)}_{13} \bigg)\\
 \therefore  f^3(t,x,D^{(1)})- f^3(t,x,D^{(2)}) &\leq F_3(t,x)\bullet \big[g_3(D^{(1)})-g_3(D^{(2)})\big]
  \end{aligned}
\end{equation*}
 
 \textbf{F3:}  Since $S,I,B$ are bounded on $[0,T]$  so $F(\bullet,x^{u}) \in \mathcal{L}_{1}$
 
 Now we have to show that the running cost function 
 $$C(t,x,D)=I(t) + B(t)+P\big[D^{2}_{11}(t)+D^{2}_{12}(t)+D^{3}_{13}(t)\big]
   +Q\big[D^{2}_{21}(t)+D^{2}_{22}(t)+D^{3}_{23}(t)\big] + R\big[|D_{3}(t)|^{2}\big]$$
satisfies the  conditions $(C1)-(C5)$ of  \textit{Theorem 2.2} of \cite{boyarsky1976existence}. Here $C:[0,T]\times X \times U \to \mathbb{R}$\\

\textbf{C1:} Here  $C(t,\bullet,\bullet)$ is a continuous function  as it is sum of continuous functions which are functions of $t\in [0,T]$.\\
\textbf{C2:} Since a $S,I$ and $B$  and all $D_{ij}$'s are bounded implying that $C(\bullet,x,D)$ is bounded and hence measurable for each $x\in X$ and $D_{ij} \in U$.\\ 
\textbf{C3:} Consider $\Psi(t) = \kappa$ such that $ \kappa = \min \{I(0),B(0)\} $ then $\Psi$ will bounded such that for all $t\in [0,T]$, $x \in X$ and $D_{ij}\in U$ we have 
$$C(t,x,D)\geq \Psi(t)$$
\textbf{C4:} Since $C(t,x,D)$ is sum of the function  which are convex in $U$ for each fixed $(t,x)\in [0,T]\times X $  therefore $C(t,x,D)$ follows the same.\\
\textbf{C5:} Using similar type of argument we can easily shoe that for each fixed $(t,x)\in [0,T]\times X $, $C(t,x,D)$ is a monotonically increasing function.

Hence we have shown that the optimal control problem satisfies the all hypothesis of the \textit{Theorem 2.2} of \cite{boyarsky1976existence}.
Therefore there exists a $8$- tuple of optimal controls  $\big(D_{1}^*(t),D_{2}^*(t),D_{3}^*(t)\big)$ in the set of admissible controls $U$ such that the cost function is minimized.
\end{proof}

{\subsection{\bf{Characteristics for the optimal control}}} \vspace{.2cm}

 In this section we obtain the characteristics of the optimal control using the 
  \textit{Pontryagin’s Maximum Principle}  \cite{liberzon2011calculus}.

 The Hamiltonian for the system (\ref{opti}) - (\ref{sec11equ3}) is given by
\begin{equation}
     H\big(I,V, D_1,D_2,D_3,\lambda\big) = I(t) + B(t) + P\big[D^{2}_{11}(t)+D^{2}_{12}(t)+D^{3}_{13}(t)\big] + Q\big[|D_{2}|^{2}\big] + R\big[|D_{3}|^{2}\big] + \lambda_{1}\frac{dS}{dt} + \lambda_{2}\frac{dI}{dt} + \lambda_{3}\frac{dB}{dt}
\end{equation}
where $\lambda = \big(\lambda_1,\lambda_2,\lambda_3\big)$ is  the co-state vector or adjoint vector. Now the {canonical equations} that relates state variable and co state variable are given by

\begin{equation}
  \begin{aligned}
  \frac{d\lambda_1}{dt} &=& -\dfrac{\partial H}{\partial S} \\
  \frac{d\lambda_2}{dt} &=& -\dfrac{\partial H}{\partial I} \\
  \frac{d\lambda_3}{dt} &=& -\dfrac{\partial H}{\partial B}
  \end{aligned}
\end{equation}

 Now substituting the value of the Hamiltonian the above equation we get
\begin{equation}
  \begin{aligned}
  \frac{d\lambda_1}{dt} &= \big(\beta B +\mu_1+\gamma+D_{11}+D_{21}-D_{31}\big)\lambda_{1}-\big(\beta B\big) \lambda_2\\
  \frac{d\lambda_2}{dt} &= \big(\mu_1+\delta+D_{12}+D_{22}\big)\lambda_2-\big(\alpha-D_{23}^2-D_{33}\big)\lambda_{3}-1 \\
  \frac{d\lambda_3}{dt} &= \big(\beta S)\lambda_{1}-\big(\beta S)\lambda_{2}+\big(y+\mu_{2}+D_{13}^{2}\big)\lambda_{3}-1
  \end{aligned}
\end{equation}
along with the transversality condition $\lambda_{1}(T)=0$, $\lambda_{2}(T)=0$ and $\lambda_{3}(T)=0$. Now using the fact that at optimal controls, $D_{ij}=D_{ij}^{*}$ and the value of Hamiltonian is minimum implying that  $\dfrac{\partial H}{\partial D_{ij}} = 0$ at $D_{ij}=D_{ij}^{*}$ for $1\leq i,j \leq 3$ and $ij\neq 32$,
and solving (7.7)  we have the following values for the optimal controls.

$$D_{11}^* = min \Bigg\{max\bigg\{\dfrac{S\lambda_{1}}{2P},0\bigg\},D_{11}max\Bigg\}$$
$$D_{12}^* = min\Bigg\{max\bigg\{\dfrac{I\lambda_{2}}{2P},0\bigg\},D_{12}max\Bigg\}$$
$$D_{13}^* = min\Bigg\{max\bigg\{\dfrac{2I\lambda_{3}}{3P},0\bigg\},D_{13}max\Bigg\}$$
$$D_{21}^* = min\Bigg\{max\bigg\{\dfrac{S\lambda_{1}}{2Q},0\bigg\},D_{21}max\Bigg\}$$
$$D_{22}^* = min\Bigg\{max\bigg\{\dfrac{I\lambda_{2}}{2Q},0\bigg\},D_{22}max\Bigg\}$$
$$D_{23}^* = min\Bigg\{max\bigg\{\dfrac{2B\lambda_{3}}{3Q},0\bigg\},D_{23}max\Bigg\}$$
$$D_{31}^* = min\Bigg\{max\bigg\{\dfrac{-S\lambda_{1}}{2R},0\bigg\},D_{31}max\Bigg\}$$
$$D_{33}^* = min\Bigg\{max\bigg\{\dfrac{I\lambda_{2}}{2R},0\bigg\},D_{33}max\Bigg\}$$

{\subsection{\bf{Numerical Studies for the  Optimal Control Problem with MDT}}} \vspace{.2cm}

\hspace{0.5 in} In this section we numerically obtain the optimal drug regimen for the control problem (\ref{opti}) - (\ref{sec11equ3}) using the optimal controls obtained in the earlier section.

\hspace{0.5 in} For the  numerical simulations we consider a time period of $100$ days 
($T=100$) and the parameter values are chosen as  $\omega = 20.9$, $\beta = 0.03$, $\mu_1 = 0.00018$,
$\gamma = 0.01795$, $\delta = 0.2681$, $\alpha = 0.2$, $y = 0.3$ and $\mu_2 = 0.57$. First
we have solved the system numerically without any drug intervention. All the numerical 
calculation were done in MATLAB  and we used $4^{th}$ order {Runge-Kutta} method to
 solve system of ODEs. Here we consider the initial value of the state variables
as $S(0)=520$, $I(0)=275$ and $B(0)=250$  as in \cite{ghosh2021mathematical}.

\hspace{0.5 in} Further to simulate the system with controls, we use the
{Forward-backward sweep} method starting with the initial value of the controls as
zero and estimate the sate variables forward in time. Since the the transversality
conditions have the value of adjoint vector at end time $T,$ so the adjoint vector was calculated backward in time.

\hspace{0.5 in} Using the value of state variables and adjoint vector we
 calculate the  control variables at each time instance that get
updated in each iteration. We continue this till the convergence criterion is met
\cite{lenhart2007optimal}.

\hspace{0.5 in}  The weights  $P, Q$ and
$R$  in the cost function $\mathcal{J}_{min}$ are chosen based on their  \textit{hazard ratio} of the corresponding drugs. We chose the weights directly proportional to the hazard ratios.  In {Table} \ref{HR} the {hazard ratios} of the
different drugs are enlisted. We have chosen the weights$(P, Q$
and $S)$ proportional to  the hazard ratios i.e. $P = 1$, $Q = 1.99$ and $R = 7.1$.

\begin{table}[ht!]   
\centering 
\begin{tabular}{|c|c|c|} 
\hline

\textbf{Drugs} &  \textbf{Hazard Ratio } & \textbf{Source} \\ 

\hline\hline
Rifampin & 0.26 & \cite{bakker2005prevention} \\

\hline\hline
Dapsone & 0.99 & \cite{cerqueira2021influence}  \\

\hline\hline
Clofazimine & 1.85 & \cite{cerqueira2021influence}\\

\hline

\end{tabular}
\caption{Hazard Ratio of the drugs }
\label{HR}
\end{table}

We now numerically simulate the $S, I $ and $B$ populations without control interventions, with single control intervention, with two control interventions and finally with three control interventions of MDT.

\begin{figure}[ht!]
    \centering
    \includegraphics[height = 7cm, width =16cm]{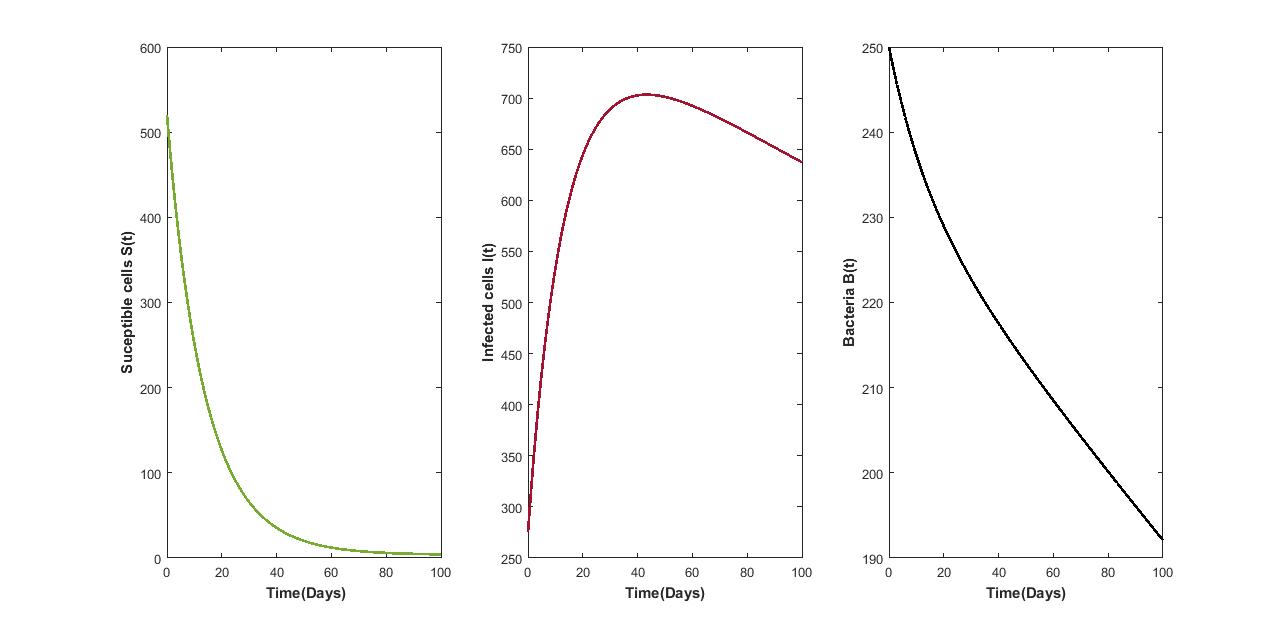}

    \caption{Plots depicting the  $S, I $ and $B$ populations without any control interventions }
    \label{wctrl}
\end{figure} 

\vspace{.3cm}

\hspace{0.5 in} The  {Figure} \ref{wctrl} depicts  the dynamics of the $S,  I $ and $B$ populations without any control/drug interventions

\newpage

\begin{figure}[ht!]
    \centering
    \includegraphics[height = 7cm, width = 16cm]{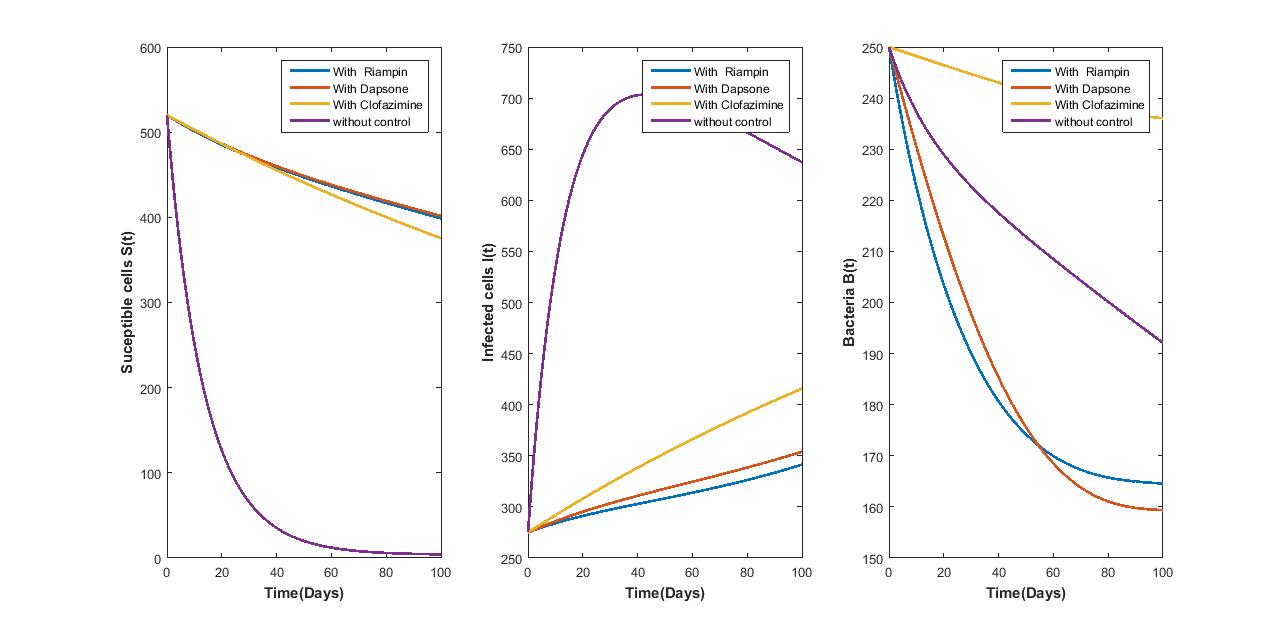}

    \caption{Plots depicting the  dynamics of the $S, I $ and $B$ populations  when one drug is introduced}
    \label{onedrug}
\end{figure}

\begin{figure}[ht!]
    \centering
    \includegraphics[height = 7cm, width =16cm]{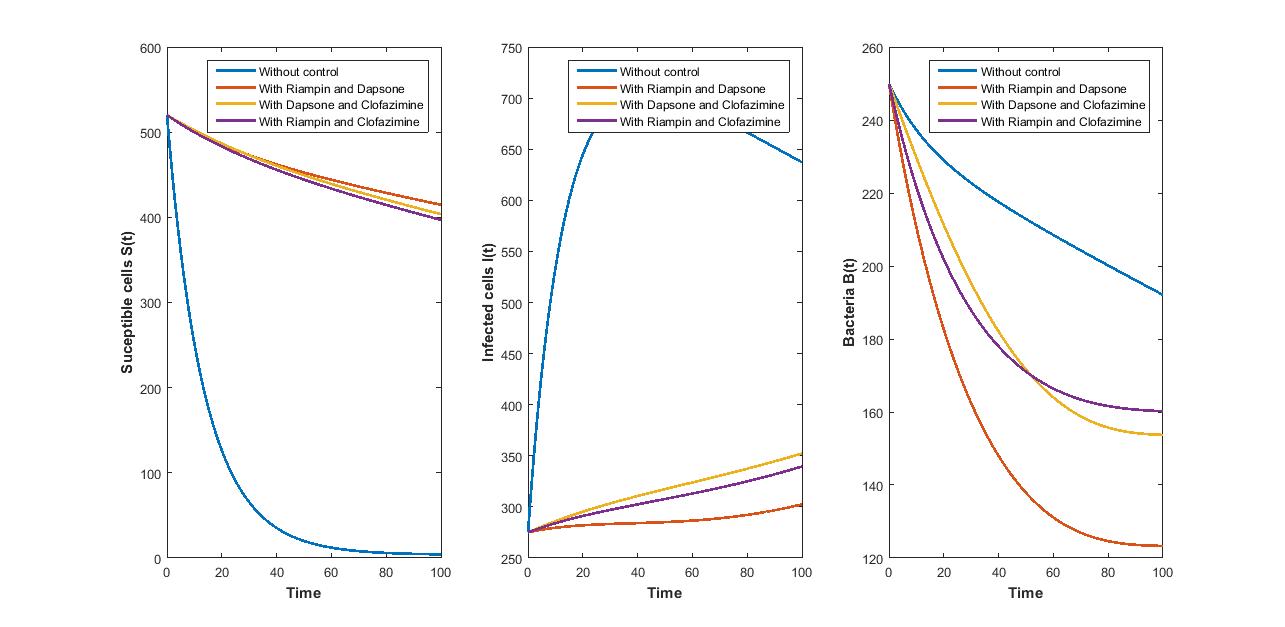}

    \caption{Plots depicting the  dynamics of  the $S, I $ and $B$ populations  when combination of two drug is introduced}
    \label{twodrug}
\end{figure}

\vspace{.2cm}

\hspace{0.5 in} The plot \ref{onedrug} illustrate that when individually drugs are administered the {susceptible cell} count decrease  and the opposite effect is seen for {infected } cells  and {bacterial load } compartments. One notable thing is there that the {clofazimine} alone can't decrease the {bacterial load } in the long run.  {Figure} \ref{twodrug} shows that the combination of two drugs are more effective than one drug given at a time. As earlier here as we see the {susceptible cell} count decrease and increase in both the {infected } cells  and {bacterial load } compartments.

\newpage

\begin{figure}[ht!]
    \centering
    \includegraphics[height = 7cm, width =16cm]{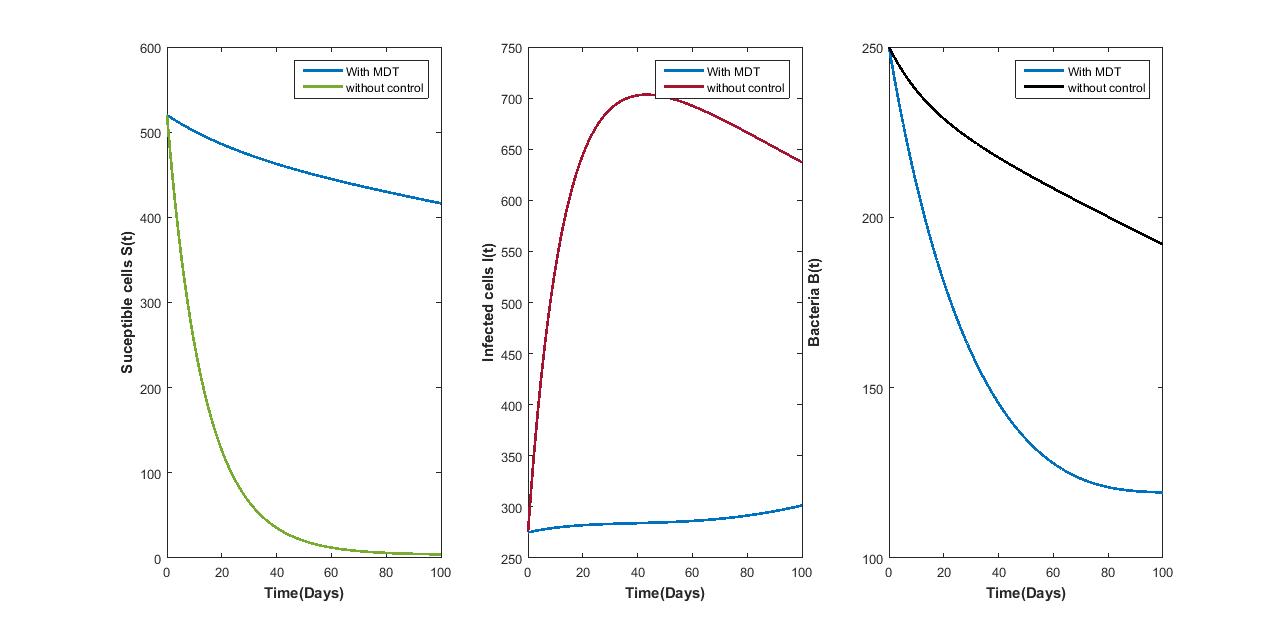}
    \caption{Plots depicting the  dynamics of the the $S, I $ and $B$ populations  with MDT intervention}
    \label{wmdt}
\end{figure}

\vspace{.2cm}

\hspace{0.5 in} The {Figure}  \ref{wmdt} shows the dynamics of 
the $S,  I $ and $B$ populations with MDT intervention whose findings are in similar lines to earlier two plots.

The following  {Table} \ref{avg}  gives the average $S, I $ and $B$  cell count for single drug, two drug combination and MDT scenarios. From the table it can be seen that MDT is the best and optimal combination for achieving the optimal increase in susceptible cells and optimal decrease in both infected cells and bacterial load.

\vspace{.2cm}

\begin{table}[ht!]   
\centering 
\begin{tabular}{|c|c|c|c|} 
\hline

\textbf{Drug Combination} &  \textbf{Avg susceptible cells} & \textbf{Avg Infected cells} & \textbf{Avg Bacterial load } \\ 

\hline\hline
Rifampin & 450.845778 & 308.544767 & 184.353265  \\

\hline\hline
Dapsone & 452.654249 & 316.861762 & 185.826321 \\

\hline\hline
Clofazimine & 443.133511 & 350.173515 & 241.995028 \\

\hline\hline
Rifampin and Dapsone & 457.141441 & 286.714732 & 153.050429  \\

\hline\hline
Rifampin and Clofazimine & 453.456572 & 316.303543 & 182.278864  \\

\hline\hline
Dapsone and Clofazimine & 448.689818 & 307.856580 & 181.470695  \\

\hline\hline
MDT & 457.899776 & 286.431294 & 150.360779  \\

\hline

\end{tabular}
\caption{ Average count of the $S, I $ and $B$  cells for single drug, two drug combination and MDT scenarios }
\label{avg}
\end{table}

\vspace{.3cm}

\subsection{{\bf{Optimal control problem associated with MDT along with steroids}}}

\vspace{.2cm}

\hspace{0.5in} {\textit{Corticosteroid}} is a steroid which is mainly used for
protecting the nerve damage by suppressing the cytokines responses caused
due to presence of \textit{M. leprae} \cite{shetty2010effect}. Corticosteroid is usually given after some days of MDT drugs. To capture this aspect we introduce a time delay $\tau$ in the MDT control. In others words we consider $D_{ij}'s$ at $(t-\tau)$  and consider the control associated with steroid as $C(t)$. 

\hspace{0.5 in} With the above modifications, the  set of controls now is given by 

$$U=\Big\{D_{ij}(t):D_{ij}(t)\in[0,D_{ij}max],C(t)\in [0,Cmax] 1\leq i,j \leq 3,ij \neq 32, t\in[0,T]\Big\}$$

\hspace{0.5 in} and the modified objective function and control system is given by

\begin{equation}
  \begin{aligned}
   \mathcal{J}_{min}\big(D_{1},D_{2},D_{3}\big) &= \int_{0}^{T} \Big(I(t) + B(t)+P\big[D^{2}_{11}(t-\tau)+D^{2}_{12}(t-\tau)+D^{3}_{13}(t-\tau)\big]\\
   & +Q\big[D^{2}_{21}(t-\tau)+D^{2}_{22}(t-\tau)+D^{3}_{23}(t-\tau)\big] + R\big[|D_{3}(t-\tau)|^{2}\big] + T C^2(t) \Big) dt 
  \label{opti2}
  \end{aligned}
\end{equation}

\begin{eqnarray}
   	\frac{dS}{dt}& =&  \omega \ - \beta SB  - \gamma S - \mu_{1} S -D_{11}(t-\tau)S - D_{21}(t-\tau)S+D_{31}(t-\tau)S+C(t)S\label{sec11.3equ1} \\
   	\frac{dI}{dt} &=& \beta SB \ -\delta I - \mu_{1} I-D_{12}(t-\tau)I - D_{22}(t-\tau)I  \label{sec11.3equ2}\\ 
   	\frac{dB}{dt} &=&  \big(\alpha-D_{23}^2(t-\tau)-D_{33}(t-\tau)\big) I  \ - y B    -  \mu_{2} B-D_{13}^2(t-\tau)B \label{sec11.3equ3}
\end{eqnarray} 

 Here the  the Lagrangian  is the integrand of the cost function   (\ref{opti2}) and is given by 
\begin{equation}
   \begin{aligned}
    L\big(I,V, D_1,D_2,D_3,C\big)&= \Big(I(t) + B(t)+P\big[D^{2}_{11}(t-\tau)+D^{2}_{12}(t-\tau)+D^{3}_{13}(t-\tau)\big]\\
   & +Q\big[D^{2}_{21}(t-\tau)+D^{2}_{22}(t-\tau)+D^{3}_{23}(t-\tau)\big] + R\big[|D_{3}(t-\tau)|^{2}\big] + T C^2(t) \Big)
   \end{aligned}
\end{equation}

 The admissible set of solutions for the above optimal control problem  will now lie in the set

$$\Omega =\Big\{\big(I,V,D_1,D_2,D_3,C\big): I,V satisfy \hspace{0.1in} (\ref{sec11equ1})-(\ref{sec11equ3}) \  \forall \big(D_1,D_2,D_3,C\big)\in U\Big\}$$

The existence of the optimal control can be shown in the similar way as it was shown in the previous optimal control problem in the preceding section.

We see that the  Hamiltonian for the system (\ref{opti2}) - (\ref{sec11.3equ3}) is given by

\begin{equation}
     H\big(I,V, D_1,D_2,D_3,\lambda\big) = L\big(I,V, D_1,D_2,D_3,C\big)+ \lambda_{1}\frac{dS}{dt} + \lambda_{2}\frac{dI}{dt} + \lambda_{3}\frac{dB}{dt}
\end{equation}
where $\lambda=\big(\lambda_1,\lambda_2,\lambda_3\big)$ is  the co-state vector or adjoint vector. Now the \textit{canonical equations} that relates state variable and co state variable are given by

\begin{equation}
  \begin{aligned}
  \frac{d\lambda_1}{dt} &=& -\dfrac{\partial H}{\partial S} \\
  \frac{d\lambda_2}{dt} &=& -\dfrac{\partial H}{\partial I} \\
  \frac{d\lambda_3}{dt} &=& -\dfrac{\partial H}{\partial B}
  \end{aligned}
\end{equation}

Now substituting the value of the Hamiltonian in the above equation we get
\begin{equation}
  \begin{aligned}
  \frac{d\lambda_1}{dt} &= \big(\beta B +\mu_1+\gamma+D_{11}(t-\tau)+D_{21}(t-\tau)-D_{31}(1-\tau)-C(t)\big)\lambda_{1}-\big(\beta B\big) \lambda_2\\
  \frac{d\lambda_2}{dt} &= \big(\mu_1+\delta+D_{12}(t-\tau)+D_{22}(t-\tau)\big)\lambda_2-\big(\alpha-D_{23}^2(t-\tau)-D_{33}(t-\tau)\big)\lambda_{3}-1 \\
  \frac{d\lambda_3}{dt} &= \big(\beta S)\lambda_{1}-\big(\beta S)\lambda_{2}+\big(y+\mu_{2}+D_{13}^{2}(t-\tau)\big)\lambda_{3}-1
  \end{aligned}
\end{equation}
\hspace{0.5 in}along with the transversality condition $\lambda_{1}(T)=0$, $\lambda_{2}(T)=0$ and $\lambda_{3}(T)=0$. 

We now have  $\frac{\partial H}{\partial D_{ij}}=0$ and $\frac{\partial H}{\partial C}=0$ at $D_{ij}=D_{ij}^{*}$ and $C=C^*$ for $1\leq i,j \leq 3$ and $ij\neq 32$.\\

\hspace{0.5 in}Now  differentiating the {Hamiltonian} and solving it for $D_{ij}^{*}$ and $C^*$ we have the values  for the optimal controls as

$$D_{11}^*(t-\tau) = min \Bigg\{max\bigg\{\dfrac{S\lambda_{1}}{2P},0\bigg\},D_{11}max\Bigg\}$$
$$D_{12}^*(t-\tau) = min\Bigg\{max\bigg\{\dfrac{I\lambda_{2}}{2P},0\bigg\},D_{12}max\Bigg\}$$
$$D_{13}^*(t-\tau) = min\Bigg\{max\bigg\{\dfrac{2I\lambda_{3}}{3P},0\bigg\},D_{13}max\Bigg\}$$
$$D_{21}^*(t-\tau) = min\Bigg\{max\bigg\{\dfrac{S\lambda_{1}}{2Q},0\bigg\},D_{21}max\Bigg\}$$
$$D_{22}^*(t-\tau) = min\Bigg\{max\bigg\{\dfrac{I\lambda_{2}}{2Q},0\bigg\},D_{22}max\Bigg\}$$
$$D_{23}^*(t-\tau) = min\Bigg\{max\bigg\{\dfrac{2B\lambda_{3}}{3Q},0\bigg\},D_{23}max\Bigg\}$$
$$D_{31}^*(t-\tau) = min\Bigg\{max\bigg\{\dfrac{-S\lambda_{1}}{2R},0\bigg\},D_{31}max\Bigg\}$$
$$D_{33}^*(t-\tau) = min\Bigg\{max\bigg\{\dfrac{I\lambda_{2}}{2R},0\bigg\},D_{33}max\Bigg\}$$
$$C^*(t) = min\Bigg\{max\bigg\{\dfrac{S\lambda_{1}}{2T},0\bigg\},D_{33}max\Bigg\}$$

\subsubsection{Numerical simulations for  Optimal Control with both MDT and Steroids} 

 \hspace{0.5 in} Here we use all the parameter values and initial conditions  same as  in the previous optimal control problem. The value of the weight $T$ was chosen to be $6.4230$ based on the  hazard ratio  value $1.67$  \cite{cerqueira2021influence}. Instead of forward backward sweep we use only forward sweep for calculating the state variables and adjoint vectors after the delay $\tau$.  Here we considered $\tau= 55 $  days and the step size as $h = 0.0000045$. 
 
 
 \begin{figure}[ht!]
    \centering
    \includegraphics[height = 7cm, width =16cm]{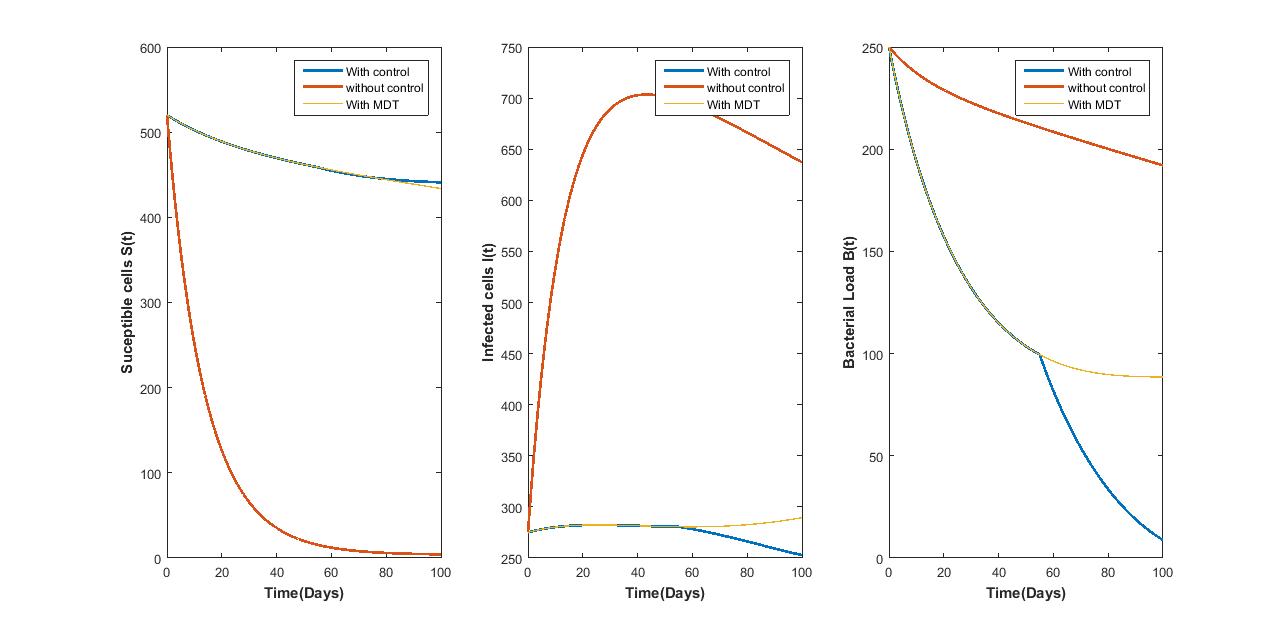}
    \caption{Plots depicting the system dynamics when MDT and steroids are  intervened.}
    \label{wster}
\end{figure}

From the Figure \ref{wster} it can be seen that the combined combination of MDT and  corticosteroid seems to be doing the best job in decreasing the lepra type 1 reaction disease burden.

\section{Comparative and  Effectiveness Study} \vspace{.2cm} 

\hspace{0.5 in} In this section we will perform the comparative
and effectiveness study for the system (\ref{sec11equ1}) -  (\ref{sec11equ3}).  

\hspace{0.5 in} For this system without any control/drug interventions the basic reproduction number is given by 

$$\mathcal{R}_{0}=\frac{\alpha \beta \omega}{(\gamma+\mu_{1})(\delta+\mu_{1})(y+\mu_{2})}$$

\hspace{0.5 in} Now to study the effectiveness of each of these control/drug interventions we calculate the modified reproduction number $\overline{\mathcal{R}_{0}}$  based on the modified parameters which gets altered owing to these interventions as follows: 

\begin{itemize}
    \item The drug dapsone primarily acts on the inhibition of viral replication. Based on this we consider $\alpha$ to be $\alpha(1-\epsilon)$ where $\epsilon$ denotes the efficiency of the drug dapson.

    \item Since the drug {rifampin} is a killer of bacteria it indirectly
  reduces the interaction between {susceptible} cells and the
  {bacteria}. Owing to this we choose $\beta$  as    $\beta(1-\rho)$ where $\rho$ denotes the efficacy of {rifampin} in killing bacteria.

    \item The drug clofazimine primarily inhibits the cytokines responses indirectly reducing the death of healthy cells. Owing to this we consider $\gamma$ to be 
    $\frac{\gamma}{(1-c)}$ where $c$ denotes the efficacy of {clofazimine} in supressing cytokines responses.

\end{itemize}

\hspace{0.5 in} With the above modified parameters based on the action of control/drug interventions, we get the modified reproduction number $\overline{\mathcal{R}_{0}}$  as

$$\overline{\mathcal{R}_{0}}=\dfrac{\alpha (1-\epsilon)\beta (1-\rho) \omega}{\Big(\dfrac{\gamma}{1-c}+\mu_{1}\Big)(\delta+\mu_{1})(y+\mu_{2})}$$

\hspace{0.5 in} We now do the comparative and  effectiveness study  by calculating the percentage of reduction  $\mathcal{R}_{0}$ with reference to modified  $\overline{\mathcal{R}_{0}}$ as follows:

$$Percentage \hspace{0.1 in} of \hspace{0.1 in} reduction \hspace{0.1 in} in \hspace{0.1 in} \mathcal{R}_{0}=\Bigg[\frac{\mathcal{R}_{0}-\overline{\mathcal{R}_{0}}}{\mathcal{R}_{0}}\Bigg]\times 100 $$

\hspace{0.5 in} We do this study for  different efficacy levels of the drugs such as

 (a) Low Efficacy (LE) given by  $0.3$ (b) Medium Efficacy (ME)  given by $0.6$ and (c) High Efficacy (HE) given by  $0.9$.

\hspace{0.5 in} In the following table the comparative and effectiveness study is done and the the drug combinations are ranked based on the reduction in percentage of ${\mathcal{R}_{0}}$ for different efficacy levels of the drugs. The highest rank is given for the drug combination that has highest reduction in the reproduction number.  The efficacy at different levels were chosen with {rifampin}  taken as the
base value  and the efficacy of {dapsone} and
{clofazimine} were taken lesser than this based
on their hazard ratios using the fact that higher the
hazard ratio lower the efficacy level. \\

\begin{table}[ht!]   
\centering 
\begin{tabular}{|c|c|c|c|c|c|c|c|} 
\hline

\textbf{Sl No} &  \textbf{Drug Combination } & \textbf{$\%$age LE } & \textbf{Rank}&  \textbf{$\%$age ME } & \textbf{Rank}&  \textbf{$\%$age HE } & \textbf{Rank}  \\ 

\hline\hline
1 & Rifampin & 30.000000 & 4 & 60.000000 & 4 & 90.000000 & 4 \\

\hline\hline
2 & Dapsone & 7.880000 & 2 & 15.750000 & 2 & 23.630000 & 2 \\

\hline\hline
3 & Clofazimine & 0.043724 & 1 & 0.091317 & 1 & 0.143575 & 1\\

\hline\hline
4 & Rifampin and Dapsone & 35.516000 & 6 & 66.300000 & 6 & 92.363000 & 6 \\

\hline\hline
5 & Rifampin and Clofazimine & 30.030607 & 5 & 60.036527 & 5 & 90.014357 & 5\\

\hline\hline
6 & Dapsone and Clofazimine & 7.920279 & 3 & 15.826935 & 3 & 23.739648 & 3 \\

\hline\hline
7 & MDT & 35.544195 & 7 & 66.330774 & 7 & 92.373965 & 7\\

\hline

\end{tabular}

\caption{ Comparative and  effectiveness study in terms of ranking  for  different combinations of drug interventions  for Low efficacy (LE), Medium efficacy (ME) and High efficacy (HE)    }
\label{eff}
\end{table}  \vspace{.2cm}

From the above Table \ref{eff} dealing with the comparative and effectiveness study it can be seen that MDT treatment seems to be working the best in reducing the ${\mathcal{R}_{0}}$ percentage in comparison to single drug and two drug combinations. These findings are in line with the conclusion made for MDT interventions in section 7.4 in the optimal control setting.  \\

\section{Discussions and Conclusions} \vspace{.2cm}

\hspace{0.5 in} Based on the pathogenesis of leprosy in this work we have framed an deterministic model dealing with the type - I  lepra reaction  and the causation biomarkers . We initially studied the entire natural history of this model.
 The findings from this study include the following. The proposed system admits two steady dynamic states one being  disease-free
equilibrium and the other being the infected equilibrium. For  $\mathcal{R}_{0} < 1$ the system  tends to stabilized around the disease free equilibrium and for $\mathcal{R}_{0} > 1$ the
system tends to stabilize around the infected equilibrium. The system undergoes a trans-critical bifurcation at $\mathcal{R}_{0} = 1$.   This developed model was validated through the 2D heat plot based on the characteristic of average doubling time of the {\textit{ M.Laprae. }} The
sensitivity analysis using  PRCC and SRCC methods showed that the burst rate of the bacteria $\alpha,$
is the most sensitive parameter and in case of combination of two parameters,
the rate of death of infected cells due to cytokines $\delta,$ in combination with $\alpha$ seemed to be the most sensitive parameter combination. \\

\hspace{0.5 in}  After the natural history, we studied two optimal control problems the first dealing with the MDT interventions and second dealing with MDT along with steroid interventions. The findings from these studies include the following. For individual drug intervention scenario, the drug {rifampin} has the highest impact in reducing both the infected cells and the bacterial load. For the two drug combinations scenario, 
{rifampin} along  {dapsone} combination was the best in reducing the disease burden.  Finally we concluded that MDT combination drug intervention was the best in reducing the disease burden in comparison with single and two drug combinations.  The {Table} \ref{avg} summarizes and justifies the above findings.
Further the optimal control problem dealing with MDT along with steroid interventions also led to the conclusion that the optimal intervention is the combined intervention of administering MDT along with steroid intervention.
 The findings from the  comparative and  effectiveness study
show that the drug {clofazimine} has the least impact in reducing the disease burden when applied individually and the drug rifampin has the highest impact. Overall MDT intervention does the best job in reducing the disease burden. The findings from  the comparative and  effectiveness study are in line with the observations of the optimal control studies. \\

\hspace{0.5 in} This within-host modeling study of type - I  lepra reaction  involving the crucial biomarkers   is a first of its kind.  The finding from this novel
and comprehensive study will hep the clinicians and public health researchers in early detection of lepra reactions through study of biomarkers for prevention of subsequent disabilities.

\vspace{.2cm}

\printbibliography

 
 
\end{document}